\documentclass{amsart}

\usepackage[applemac]{inputenc}
\usepackage[english]{babel} 
\usepackage{graphicx}

\usepackage{amssymb}
\usepackage{amsfonts}
\usepackage{amscd}
\usepackage{amsthm}
\usepackage{epstopdf}

\usepackage[colorlinks,urlcolor=blue,citecolor=blue,filecolor=blue,linkcolor=blue]{hyperref}

\newtheorem{theorem}{Theorem}[section]
\newtheorem{corollary}[theorem]{Corollary}

\newtheorem{lemma}[theorem]{Lemma}
\newtheorem{proposition}[theorem]{Proposition}

\newtheorem{example}[theorem]{Example}

\newtheorem{definition}[theorem]{Definition}

 \newtheorem*{thm*}{Theorem}
 \newtheorem*{rem*}{Remark}

\theoremstyle{definition}
\theoremstyle{remark}
\numberwithin{equation}{section}

\newcommand{\GG}{{\mathbb G}}

\newcommand{\reals}{\mathbb{R}}
\newcommand{\R}{\mathbb{R}}
\newcommand{\naturals}{\mathbb{N}}

\newcommand{\dist}{\mathrm{dist}}

\newcommand{\Tan}{\mathrm{Tan}}
\newcommand{\LTan}{\mathrm{LTan}}
\newcommand{\pTan}{\mathrm{pTan}}
\newcommand{\pLTan}{\mathrm{pLTan}}
\newcommand{\graph}{\mathrm{graph}}

\newcommand{\Li}{\mathop{\mathrm{Li}}}
\newcommand{\Ls}{\mathop{\mathrm{Ls}}}

\newcommand{\Para}{\mathrm{P}}
\newcommand{\ang}{\mathrm{ang}}

\newcommand{\calP}{\mathcal{P}}
\newcommand{\calG}{\mathcal{G}}
\newcommand{\calJ}{\mathcal{J}}

\newcommand{\dd}{\mathrm{d}}
\newcommand{\B}{\mathrm{B}}
\newcommand{\der}{\mathrm{der}}
\newcommand{\GL}{\mathop{\mathrm{GL}}}
\newcommand{\intt}{\mathrm{int}}

\newcommand{\labelpag}[1]{\refstepcounter{enumi}\label{#1}}

\begin{document}

\title[Geometric Characterizations of $C^{1}$ manifold]{Geometric Characterizations of $C^{1}$ Manifold in Euclidean spaces by Tangent Cones}
\author{Francesco Bigolin}
\address{Dipartimento di Matematica\\
Universit\`{a} di Trento, 38050 Povo (TN), Italy}
\email{bigolin@science.unitn.it}
\author{Gabriele H. Greco}
\address{Dipartimento di Matematica\\
Universit\`{a} di Trento, 38050 Povo (TN), Italy}
\email{greco@science.unitn.it}

\dedicatory{Commemorating the 150th Birthday of Giuseppe Peano (1858-1932)}
\date{February 12, 2012}
\begin{abstract} 
A remarkable and elementary fact that a locally compact set $F$ of Euclidean space  is a smooth manifold if and only if the lower and upper paratangent cones to $F$ coincide at  every point, is proved. 
 The celebrated \textsc{von Neumann}'s result (1929) that a locally compact subgroup of  the general linear group is a smooth manifold, is  a straightforward application.
A   historical account on the subject is provided in order to enrich the mathematical panorama. Old  characterizations of smooth manifold (by tangent cones), due to \textsc{Valiron} (1926, 1927) and \textsc{Severi} (1929, 1934)  are recovered; modern characterizations, due to \textsc{Gluck} (1966, 1968),   \textsc{Tierno} (1997) and  \textsc{Shchepin} and \textsc{Repov\v{s}} (2000) are restated.
\end{abstract}

\maketitle

\section*{Introduction}
\label{sec-intro}

A primary aim of this paper is to prove that

\begin{thm*}[\textbf{Four-cones coincidence}]
 A non-empty subset $F$ of $\mathbb{R}^{n}$ is a $C^{1}$-manifold
if and only if $F$ is locally compact and the lower and upper paratangent
cones to $F$ coincide at  every point, i.e., 
\begin{equation*}
\mathrm{pTan}^{-}(F,x)=\mathrm{pTan}^{+}(F,x)\text{ for every }x\in F. 
\end{equation*}
\end{thm*}

This theorem entails numerous other existing characterizations,  
as well as \textsc{von Neumann}'s theorem (1929) that a
locally compact subgroup of the general linear group is a smooth manifold.

The \emph{upper paratangent cone }$\mathrm{pTan}^{+}(F,x)$ (introduced by 
\textsc{Severi} and \textsc{Bouligand} in 1928) and \emph{lower tangent cone}
$\mathrm{pTan}^{-}(F,x)$ (introduced by \textsc{Clarke} in 1973) are defined
respectively as the upper and the lower limits of the homothetic relation
\begin{equation*}
\tfrac{1}{\lambda}(F-y)
\end{equation*}
as $\lambda$ tends to $0$ and $y$ tends to $x$ within $F.$ They are closely
related to the \emph{upper tangent cone }$\mathrm{Tan}^{+}(F,x)$ and the 
\emph{lower tangent cone} $\mathrm{Tan}^{-}(F,x)$, defined by \textsc{Peano}
in 1887, as the upper and the lower limits of $\tfrac{1}{\lambda}(F-x)$ as $\lambda$
tends to $0.$ In general,
\begin{equation*}
\mathrm{pTan}^{-}(F,x)\subset \mathrm{Tan}^{-}(F,x)\subset \mathrm{Tan}
^{+}(F,x)\subset \mathrm{pTan}^{+}(F,x),
\end{equation*}
so that the condition of our theorem amounts to the coincidence of all these
cones.

A secondary aim  of this paper is to retrace historical information 
by direct references to mathematical papers where notions and properties
first occured to the best of our knowledge. As a consequence to this
historical concern, some geometrical characterizations of $C^{1}$-manifold by tangent cones  implement  conditions and properties
recovered from  forgotten mathematical papers of \textsc{Valiron} and 
\textsc{Severi}.

Section \ref{sec-tan-paratan}: \textbf{Tangency and paratangency}.
Investigation of $C^{1}$-manifolds involve four tangent cones approximating, already mentioned.
The upper and lower tangent cones were introduced by \textsc{Peano} 
to ground tangency on a firm basis and to establish optimality necessary
conditions; as today modern habit, \textsc{Peano} defined upper and lower
tangent cones as limits of sets. The upper tangent cone, which was recovered
41 years later by \textsc{Severi} and \textsc{Bouligand} in 1928, is known as 
\textsc{Bouligand}'s contingent cone.

Section \ref{sec-tan-paratan-diff}: \textbf{Tangency and paratangency in
traditional sense compared with differentiability}. Characterizations of
both differentiability (called today {Fréchet differentiability}) and strict
differentiability of functions on arbitrary sets (not necessary open, as a
today habit) are stated and, with a pedagogical intent, proved. They are due
essentially to \textsc{Guareschi} and \textsc{Severi}. The modern definition
of differentiability of vector functions is due to \textsc{Grassmann}
(1862), although there is a slighter imperfection. This imperfection was
corrected by \textsc{Peano} in 1887 (for scalar functions) and in 1908 (for
vector functions). The notion of strict differentiability was introduced by 
\textsc{Peano} (1892) for real functions of one real variable and by \textsc{%
Severi} (1934) for several variables.

Section \ref{sec-grassmann}: \textbf{Grassmann Exterior algebra, limits of
vector spaces and angles between vector spaces}. Following \textsc{Peano}'s 
\emph{Applicazioni geometriche} (1887), limits of sets and exterior algebra
of Grassmann are used to defined convergence of vector spaces. Exterior
algebra of Grassmann is used to associate multi-vectors to vector spaces
and, consequently, to define the notion of angle between vector spaces of
same dimension and, finally, to express convergence of vector spaces by
their angle.
In 1888 \textsc{Peano} revisited exterior algebra of Grassmann in \emph{%
Calcolo Geometrico secondo l'Ausdehnungslehre of Grassmann} and here he
introduced the terms of bi-vector, tri-vector and, more important, the
modern notion of vector space. 

Section \ref{sec-main}: \textbf{Four-cones Coincidence Theorem: local and global version}
 We state and prove our main theorem: the four-cone theorem. As its
straightforward application, we show a celebrated \textsc{von Neumann}'s
result (1929) that a locally compact subgroup of the general linear group is
a smooth manifold.
 Moreover, some corollaries of the main theorem will provide efficacious test for visual reconnaissance of $C^{1}$-manifolds.

Appendix \ref{app-von}: \textbf{Von Neumann and alternative definitions of lower tangent
cones}. We comment on \textsc{von Neumann}'s definition of Lie algebra in
his famous paper \cite[(1929)]{vonN} on matrix Lie groups, and present
alternative definitions of the lower tangent and lower paratangent cones for
encompassing \textsc{von Neumann}'s tangent vectors.

Appendix \ref{app-history}: \textbf{From Fréchet problem to modern characterizations of smooth manifold.} In 1925 \textsc{Fréchet} inquires
into existence of non-singular continuously differentiable parametric
representations of continuous curves. This problem had been a starting,
motivating and reference point for subsequent research 
by various mathematicians.   
Two basic conditions for the existence of a non-singular parametrization of
a set (either curve or surface) were given by \textsc{Valiron}: 

$\begin{array}{ll}
(*)&\text{continuously turning tangent space,  and}\\

 (**) &\text{locally injective  orthogonal projections on tangent spaces.}\\
\end{array}$

\noindent Other conditions ensuring $(*)$ and $(**)$
were given by $\textsc{Severi}$ by means of paratangency instead of
tangency. $\textsc{Valiron}$ (1926, 1927) and $\textsc{Severi}$ (1929,
1934) present, in the setting of the topological manifolds, first
geometrical characterizations of $C^{1}$-manifold by tangent cones.
From a historical point of view, an essential condition to a complete geometrical characterization of $C^{1}$%
-manifold by tangent cones, has been a solid and univocal (but not
necessary unique) definition of tangency and a $C^{1}$ version of
differentiability, the so-called strict differentiability. Finally, old  characterizations of smooth manifold (by tangent cones), due to \textsc{Valiron} (1926, 1927) and \textsc{Severi} (1929, 1934)  are recovered; modern characterizations, due to \textsc{Gluck} (1966, 1968),   \textsc{Tierno} (1997) and  \textsc{Shchepin} and \textsc{Repov\v{s}} (2000) are restated.

\begin{rem*}\emph{
 In the following sections, the symbols $\mathbb{R}$ and $\mathbb{N}$ will denote the
real and natural numbers, respectively; and $\mathbb{N}_{1}:=\{m\in\mathbb{N%
}:n\ge1\}$, $\mathbb{R}_{+}:=\{x\in\mathbb{R}: x\ge 0\}$, $\mathbb{R}%
_{++}:=\{x\in\mathbb{R}:x>0\}$. If not otherwise specified, any set will be
a subset of some finite dimensional Euclidean space $\mathbb{R}^{n}$. An open (resp. closed) ball
of center $\hat x$ and ray $\varepsilon$ will be denoted by $%
\B_{\varepsilon}(x)$ (resp. ${\overline \B}_{\varepsilon}(x)$). ${\mathbb{P}}(\mathbb{R}^{n})$ denotes the set of all
subsets of $\mathbb{R}^{n}$. The set of accumulation points of a given set $A$
and its interior are denoted by $\mathrm{der}(A)$ and $\mathrm{int}(A)$,
respectively.
}
\end{rem*}

\section{Tangency and paratangency}\label{sec-tan-paratan}

Let $F$ be an arbitrary subset of Euclidean space $\reals^{n}$ and let $x\in \reals^{n}$.  We will  consider four types of tangent cone to $F$ at $x$:  the \emph{lower}  and the \emph{upper tangent cones}
 \[\Tan^{-}(F,x),\qquad \Tan^{+}(F,x)
 \]
 respectively; and  
the \emph{lower} and the \emph{upper paratangent cones}
\[\pTan^{-}(F,x),\qquad \pTan^{+}(F,x)
 \]
respectively. All of them are cones\,\footnote{In the sequel, a set $A\subset\reals^{n}$ is said to be a \emph{cone}, if $\lambda v\in A$ for every $v\in A$ and  each $\lambda\in\reals_{+}$.} of $\reals^{n}$. They   satisfy the following set inclusions:
\begin{equation}
\pTan^{-}(F,x)\subset \Tan^{-}(F,x)%
\subset \Tan^{+}(F,x)\subset \pTan^{+}(F,x).  \label{4cones}
\end{equation}
The elements of $\pTan^{-}(F,x)$ (resp. $\Tan^{-}(F,x)$, $\Tan^{+}(F,x)$,  $\pTan^{+}(F,x)$) are referred to as \emph{lower paratangent} (resp. \emph{lower tangent}, \emph{upper tangent}, \emph{upper paratangent}) \emph{vectors to $F$ at $x$}.

In order to define them as   lower or upper limits of homothetic sets,  let us introduce two types of limits of sets (the so-called \emph{Kuratowski limits}). Let $A_{\lambda}$ be a subset of $\reals^{n}$ for every real number $\lambda>0$. The \emph{lower limit} $\Li_{\lambda\to0^{+}}A_{\lambda}$ and \emph{upper limit}
$\Ls_{\lambda\to0^{+}}A_{\lambda}$ are defined by \footnote{The limits of sets \eqref{lowerlimit} and \eqref{upperlimit} were introduced by 
\textsc{Peano}: the lower limit in \emph{Applicazioni geometriche} \cite[(1887),  p.\,302]{peano87} and the upper limit 
in \emph{Lezioni di analisi infinitesimale} \cite[(1893), volume 2, p.\,187]{peano1893}
(see \textsc{Dolecki}, \textsc{Greco} \cite[(2007)]{DG-Peano} for further historical details).}
\begin{eqnarray}
 \Li_{\lambda\rightarrow 0^{+}}A_{\lambda}:=\{v\in 
\mathbb{R}^{n}:\lim_{\lambda\rightarrow 0^{+}}\dist(v,A_{\lambda})=0\},  \label{lowerlimit} \\
\Ls_{\lambda\rightarrow 0^{+}}A_{\lambda}:=\{v\in \mathbb{R}^{n}:\liminf_{\lambda%
\rightarrow 0^{+}}\dist(v,A_{\lambda})=0\},  \label{upperlimit}
\end{eqnarray}%
where define $\dist(x,A):=\inf\{\|x-a\|:a\in A\}$ for every $x\in\reals^{n}$ and $A\subset\reals^{n}$. 

Obviously $\Li_{\lambda\rightarrow 0^{+}}A_{\lambda}\subset\Ls_{\lambda\rightarrow 0^{+}}A_{\lambda}$. They can be characterized in terms of sequences:
\begin{eqnarray}\qquad
 v\in\Li_{\lambda\rightarrow 0^{+}}A_{\lambda}\Longleftrightarrow \begin{cases}\forall \{\lambda_{m}\}_{m}\subset\reals_{++} \text{ with }\lambda_{m}\to0^{+}, \exists\,\{a_{m}\}_{m} \text{ with }\cr
a_{m}\in A_{\lambda_{m}} \text{ eventually, } \text{ such that }\lim_{m}a_{m}=v\end{cases} \label{lowerlimitseq} 
   \\
\qquad
 v\in\Ls_{\lambda\rightarrow 0^{+}}A_{\lambda}\Longleftrightarrow \begin{cases}\exists \{\lambda_{m}\}_{m}\subset\reals_{++} \text{ with }\lambda_{m}\to0^{+}, \exists\,\{a_{m}\}_{m} \text{ with }\cr
a_{m}\in A_{\lambda_{m}} \text{ eventually, } \text{ such that }\lim_{m}a_{m}=v.\end{cases}\label{upperlimitseq} 
\end{eqnarray}%

 The \emph{lower} and \emph{upper tangent cones} $\Tan^{-}(F,x)$ and $\Tan^{+}(F,x)$ are defined, respectively, by the following blow-up \footnote{The affine variants of the lower and upper tangent cones  \eqref{lower} and \eqref{upper} were introduced by \textsc{Peano}: the lower tangent cone in \emph{Applicazioni geometriche} \cite[(1887),  p.\,305]{peano87} and the upper tangent cone in \emph{Applicazioni geometriche} \cite[(1887) n.\,11  p.\,143-144]{peano87}  (implicitly) and in \emph{Formulario Mathematico} \cite[(1903) p.\,296]{formulaire4} (explicitly).  See \textsc{Dolecki}, \textsc{Greco} \cite[(2007)]{DG-Peano} for
further historical details. }
 \begin{eqnarray}
\Tan^{-}(F,x):= \Li_{\lambda\rightarrow 0^{+}}%
\tfrac{1}{\lambda}\left( F-x\right) ,  \label{lower} \\
\Tan^{+}(F,x):= \Ls_{\lambda\rightarrow 0^{+}}\tfrac{1%
}{\lambda}\left( F-x\right) .  \label{upper}
\end{eqnarray}%
 
Since $\dist(v,\frac{1}{\lambda}(F-x))=\frac{1}{\lambda}%
\dist(x+\lambda v,F)$,  it follows from \eqref{lowerlimit} and \eqref{upperlimit} that 
\begin{eqnarray}
v\in \Tan^{-}(F,x)& \iff&\lim_{\lambda\rightarrow 0^{+}}\frac{1}{\lambda}\dist(x+\lambda v,F)=0,\label{lowerdist}\\
v\in \Tan^{+}(F,x)& \iff&\liminf_{\lambda\rightarrow 0^{+}}\frac{1}{%
\lambda}\dist(x+\lambda v,F)=0.\label{upperdist}
\end{eqnarray}%

Therefore, in terms of sequences, from  \eqref{lowerlimitseq} and \eqref{upperlimitseq} \footnote{A new sequential definition  of the lower tangent cone was introduced in \textsc{Dolecki}, \textsc{Greco} \cite[(2007)]{DG-Peano}; see Appendix \ref{app-von}.}
\begin{eqnarray}\qquad
v\in \Tan^{-}(F,x)& \iff&\begin{cases}\forall \{\lambda_{m}\}_{m}\subset\reals_{++} \text{ with } \lambda_{m}\to0^{+},
 \\ 
\exists \{x_{m}\}_{m}\subset F\text{ such that }\lim_{m}\frac{x_{m}-x}{\lambda_{m}}=v,\end{cases}\label{lowerseq}\\
\qquad v\in \Tan^{+}(F,x)& \iff&\begin{cases}\exists \{\lambda_{m}\}_{m}\subset\reals_{++} \text{ with } \lambda_{m}\to0^{+},
 \\ 
\exists \{x_{m}\}_{m}\subset F\text{ such that }\lim_{m}\frac{x_{m}-x}{\lambda_{m}}=v.\end{cases}\label{upperseq}
\end{eqnarray}%

Generally, the lower and upper tangent cones are denominated \emph{%
adjacent} and (Bouligand) \emph{contingent} cones, respectively.  \footnote{See \textsc{Aubin} and \textsc{Frankowska} \cite{aubin-frank}, \textsc{Rockafellar} and \textsc{Wets} \cite{rock-wets}, \textsc{Murdokhovich} \cite{mordu})}

The \emph{lower} and  \emph{upper paratangent cones} $\pTan^{-}(F,x)$ and $\pTan^{+}(F,x)$ are defined, respectively, by the following blow-up
\begin{eqnarray}
\pTan^{-}(F,x):=\Li_{\lambda\rightarrow
0^{+}\atop F\ni y\rightarrow x}\tfrac{1}{\lambda}\left( F-y\right),  \label{clarke}\\
\pTan^{+}(F,x):=\Ls_{\lambda\rightarrow
0^{+}\atop F\ni y\rightarrow x}\tfrac{1}{\lambda}\left( F-y\right).  \label{para}
\end{eqnarray}%

 According to \eqref{lowerlimit} and \eqref{upperlimit}, we have
\begin{eqnarray}
v\in \pTan^{-}(F,x)&\iff&\lim_{\lambda\rightarrow 0^{+}\atop F\ni y\rightarrow x }\frac{1}{\lambda}\dist(y+\lambda v,F)=0,\\
v\in \pTan^{+}(F,x)&\iff&\liminf_{\lambda\rightarrow 0^{+}\atop F\ni y\rightarrow x}\frac{1}{%
\lambda}\dist(y+\lambda v,F)=0.
\end{eqnarray}%

Therefore, in terms of sequences, from  \eqref{lowerlimitseq} and \eqref{upperlimitseq}
\begin{eqnarray}
\qquad v\in \pTan^{-}(F,x)\Longleftrightarrow\begin{cases}\forall \{\lambda_{m}\}_{m}\subset\reals_{++} \text{ with } \lambda_{m}\to0^{+},\\
\forall \{y_{m}\}_{m}\subset F  
\text{ with }y_{m}\to x,\\ \exists\{x_{m}\}_{m}\subset F\text{ such that }\lim_{m}\frac{x_{m}-y_{m}}{\lambda_{m}}=v,\end{cases}\label{clarkeseq}\\
\qquad v\in \pTan^{+}(F,x)\Longleftrightarrow\begin{cases}\exists \{\lambda_{m}\}_{m}\subset\reals_{++}\text{ with }\lambda_{m}\to0^{+},\\\exists \{y_{m}\}_{m}\subset F
\text{ with }y_{m}\to x,\\\exists \{x_{m}\}_{m}\subset F
\text{ such that }\lim_{m}\frac{x_{m}-y_{m}}{\lambda_{m}}=v.\end{cases}\label{paraseq}
\end{eqnarray}

Generally, the upper and lower paratangent cone are called \emph{paratingent
cone} \footnote{The upper paratangent cone was introduced as a set of straight-lines by \textsc{Severi} \cite[(1928) p.\,149]{Severi-conf-RM}
and \textsc{Bouligand} in \cite[(1928) pp.\,29-30]{boul_BSMF}; see \textsc{Dolecki}, \textsc{Greco} \cite[2011]{DG-Peano2} for
further historical  details.} and \emph{Clarke tangent cone}, respectively. \footnote{The lower paratangent cone  was introduced in 1973 by  \textsc{Clarke} \cite[(1975)]%
{clarke75} and redefined in terms of sequences by \textsc{Thibault} \cite[(1976), p.\,1303]{thibault},  \textsc{Hiriart-Urruty} \cite[(1977), p.\,1381]{hiriartCRAS};  in Appendix \ref{app-von} a new sequential definition is given. The upper and lower paratangent cones were expressed  by  blow-up in \textsc{Dolecki} \cite[(1882)]{tang-diff}. }

  A  sequence $\{x_{m}\}_{m}\subset \reals^{n}$ converging to a point $x$ is called  a \emph{tangential sequence}, if there exist an infinitesimal sequence $\{\lambda_{m}\}_{m}\subset\reals_{++}$ and a \emph{non-null} vector $v$ such that $\lim_{m\to\infty}\frac{x_{m}-x}{\lambda_{m}}=v$. Moreover, a couple of  sequences $\{x_{m}\}_{m}, \{y_{m}\}_{m}\subset \reals^{n}$ converging to a same point $x$ is called  a \emph{paratangential couple}, if there exist an infinitesimal sequence $\{\lambda_{m}\}_{m}\subset\reals_{++}$ and a \emph{non-null} vector $v$ such that $\lim_{m\to\infty}\frac{x_{m}-y_{m}}{\lambda_{m}}=v$. 
The upper tangent and paratangent cones share  basic compactness properties which are essential to prove several propositions in next Section. These compactness properties are expressed in terms of sequences: \footnote{Property \eqref{comp1} was first used by \textsc{Cassina} in \cite[1930]{Cassina-Milano} to show the existence of a non-null upper tangent vector to a set at an accumulation point. Properties \eqref{comp1} and \eqref{comp2} were frequently and freely used  by \textsc{Severi} and \textsc{Bouligand} in their works; for example, they are stated in \textsc{Severi} \cite[(1931), p.\,342]{Severi-Krak}. 
}
\begin{enumerate}
\item\labelpag{comp1} every  convergent sequence with infinitely many distinct terms, admits a tangential subsequence;
\item \labelpag{comp2} every proper \footnote{We say that a couple of convergent sequences $\{x_{m}\}_{m}, \{y_{m}\}_{m}\subset \reals^{n}$ is \emph{proper} if $x_{m}\ne y_{m}$ for infinitely many $m$.} couple of convergent sequences to a same point   admits a paratangential couple of subsequences. \end{enumerate}

Analogous compactness properties for both lower tangent and lower paratangent cones do not hold. For example, for $S:= \{\frac{1}{m!}: m\in\naturals, m\ge 1\}$  one has $\pTan^{-}(S,0)=\Tan^{-}(S,0)=\{0\}$ (see Example \ref{appEx1}); therefore the sequence $\{\frac{1}{m!}\}_{m}$ does not admit subsequences having non-null lower tangent vectors.

In the following proposition we collect well-known properties on tangent cones, which are used in subsequent proofs in   Section \ref{sec-main}.
\begin{proposition} Let $S\subset\reals^{n}$ be non-empty and $\hat x\in S$. Then the following properties hold.
\begin{enumerate}
\item\labelpag{coll1} The upper paratangent cone is bilateral, i.e.,
\[\pTan^{+}(S,\hat x)=-\pTan^{+}(S,\hat x).
\]
\item\labelpag{coll2} \emph{(\textsc{Bouligand} \cite[(1932), p.\,75]{bouligand})} The upper paratangent cone is upper semicontinuous, i.e.,
\[\Ls_{S\ni x\to\hat x}\pTan^{+}(S,x)\subset \pTan^{+}(S,\hat x).
\]
\item\labelpag{coll3} \emph{(\textsc{Clarke} \cite[(1975)]{clarke75})} The lower paratangent cone is convex, i.e.,
\[\pTan^{-}(S,\hat x)\text{ is convex}.
\]

\item\labelpag{coll4}  \emph{(\textsc{Cornet} \cite[(1981)]{cornetT}, \cite[(1981)]{cornet}  for  closed sets)} 
 If $S$ is  locally compact at $\hat x$ \footnote{$S$ is said to be  \emph{locally compact at $\hat x$}, whenever there exists a compact neighborhood of $\hat x$ in $S$.}, then \[\pTan^{-}(S,\hat x)=\Li_{S\ni x\to \hat x}\Tan^{+}(S,x).\]

\item\labelpag{coll5} \emph{(\textsc{Furi} \cite[(1995), p.\,96]{furi})} $S$ is open if and only if  
\[S\text{ is locally compact and }\Tan^{+}(S,x)=\reals^{n}\text{ for all }x\in S.\]

\item\labelpag{coll5rock} \emph{(\textsc{Rockafellar} \cite[(1979), p.\,149]{rock79}  for  closed sets)} If $S$ is locally compact at $\hat x$, then 
 \[\hat x\in \intt(S) \iff\pTan^{-}(S,\hat x)=\reals^{n}.\]

\item\labelpag{coll6} \emph{(\textsc{Cassina} \cite[(1930)]{Cassina-Milano})} The point $\hat x$ is an accumulation point of $S$ if and only if 
\[\Tan^{+}(S,\hat x)\text{ contains non-null vectors}.\]

\item\labelpag{coll7ante} \emph{(\textsc{Bouligand} \cite[(1928) p.\,33]{boul_BSMF}, \cite[(1932) p.\,76-79]{bouligand})}) The orthogonal projection onto the linear hull of $\pTan^{+}(S,\hat x)$ is injective on a neighborhood of $\hat x$ in $S$. More generally, if $V$ and $W$  are vector spaces  such that $V\cap\pTan^{+}(S,\hat x)=\{0\}$ and $\reals^{n}=V\oplus W$, then there is $\varepsilon\in\reals_{++}$ such that the projection along $V$ onto $W$  is injective on $S\cap \B_{\varepsilon}(\hat x)$.

\item\labelpag{coll7} \emph{(\textsc{Bouligand} \cite[(1928) p.\,33]{boul_BSMF}, 
  \cite[(1932) p.\,76-79]{bouligand})} 
  Let $p:\reals^{n}\to\reals^{n-1}$ be a projection given by 
\[
p(x_{1},\dots,x_{n-1},x_{n}):=(x_{1},\dots,x_{n-1}).
\]
If $e_{n}:=(0,\dots,0,1)\not\in\pTan^{+}(S,\hat x)$, then there exists an open ball $\B_{\varepsilon}(\hat x)$ such that $p$ is injective on $S\cap\B_{\varepsilon}(\hat x)$ and, 
 moreover, defined  $\varphi: p(S\cap\B_{\varepsilon}(\hat x))\to\reals$ by $\varphi(x_{1},\dots,x_{n-1}):=$ ``the real number $x_{n}$ such that $(x_{1},\dots,x_{n-1},x_{n})\in S\cap \B_{\varepsilon}(\hat x)$ and $p(x_{1},\dots,x_{n-1},x_{n}):=(x_{1},\dots,x_{n-1})$'', the following property holds
\begin{equation*}
\quad\quad\quad\quad\quad\quad\quad\quad \varphi \text{ is  Lipschitz  and  }\graph(\varphi)=S\cap \B_{\varepsilon}(\hat x).\quad\quad\quad\quad\quad\quad\Box
\end{equation*}

\end{enumerate}
\end{proposition}

In addition to previous characterizations \eqref{coll5}, \eqref{coll6} of open sets by tangent cones, we have that 
\begin{enumerate}
\item\labelpag{coll8} \emph{ $S$ is open if and only if   $S$ is locally compact and }
\[\Tan^{+}(S,x)=-\Tan^{+}(S,x)\text{ \emph{and} }\LTan^{+}(S,x)=\reals^{n}\text{ \emph{for all} }x\in S. \footnote{Here and in the sequel, $\LTan^{+}(S,x)$ denotes the linear hull of tangent cone $\Tan^{+}(S,x)$. \emph{Proof of \eqref{coll8}}(only sufficiency).
Suppose by absurd that a point $\hat x\in S$ is not interior to $S$. Then, by local compactness, choose $\varepsilon\in\reals_{++}$ and $\tilde x\not\in S$ such that $\overline \B_{\varepsilon}(\hat x)\cap S$ is closed and $\|\hat x-\tilde x\|<\varepsilon/2$. Now, denote by $p(\tilde x)$ a projection of $\tilde x$ on the closed set $\overline \B_{\varepsilon}(\hat x)\cap S$. Then $0<\|\bar x-p(\bar x)\|<\varepsilon/2$, $p(\tilde x)\in  \B_{\varepsilon}(\hat x)\cap S$ and $\B_{\|\tilde x-p(\tilde x)\|}(\tilde x)\cap \B_{\varepsilon}(\hat x)\cap S=\emptyset$. Hence  the open vector half-space $H^{+}:=\{v\in\reals^{n}: \langle v,\bar x-p(\bar x)\rangle>0\}$ has no elements in common with the upper tangent cone $\Tan^{+}(S, p(\tilde x))$. Since, by the hypothesis, $\Tan^{+}(S,p(\bar x))=-\Tan^{+}(S,p(\bar x))$,  then even the opposite open vector half-space $H^{-}:=\{v\in\reals^{n}: \langle v,\bar x-p(\bar x)\rangle<0\}$ has no elements in common with  $\Tan^{+}(S, p(\tilde x))$. Therefore the vector hyperplane $H:=\{v\in\reals^{n}: \langle v,\bar x-p(\bar x)\rangle=0\}$ includes  $\Tan^{+}(S, p(\tilde x))$; in contradiction of the hypothesis $\Tan^{+}(S, \tilde x)=\reals^{n}$.
}\]
\end{enumerate}

A rich and unstable terminology deals with coincidence conditions: $\Tan^{-}(S,\hat x)=\Tan^{+}(S,\hat x)$ \footnote{If this equality holds, the set $S$  is said to be ``{derivable at $\hat x$}'' in \textsc{Aubin}, \textsc{Frankowska} \cite[(1990), p.\,127]{aubin-frank},  ``{geometrically derivable} at $\hat x$'' in \textsc{Rockafellar}, \textsc{Wets} \cite[(1998), p.\,197-198]{rock-wets}, ``tangent rgular'' in \textsc{Shchepin}, \textsc{Repov\v{s}} \cite[(2000), p.\,2117]{repovs2}. If, in addition, $\Tan^{+}(S,\hat x)$ is a vector space, the set $S$  is said to be  ``{smooth at} $\hat x$'' in \textsc{Rockafellar} \cite[(1985), p.\,173]{rock85}. },  $\pTan^{-}(S,\hat x)=\Tan^{+}(S,\hat x)$ \footnote{If this equality holds, the set $S$  is said to be  ``{tangentially regular} at $\hat x$'' in   \textsc{Aubin}, \textsc{Frankowska} \cite[(1990), p.\,127]{aubin-frank},  ``{regular} at $\hat x$'' in \textsc{Rockafellar}, \textsc{Wets} \cite[(1998), p.\,220]{rock-wets}, and  ``{Clarke regular} at $\hat x$'' in \textsc{Mordukhovich} \cite[(2006), p.\,136]{mordu}.}, $\pTan^{-}(S,\hat x)=\pTan^{+}(S,\hat x)$ \footnote{If this equality holds, the set $S$  is said to be  ``{strictly smooth at} $\hat x$'' in \textsc{Rockafellar} \cite[(1985), p.\,173]{rock85}.}.
Below we present unidimensional examples  of all possible coincidence the various coincidence conditions.
\[
\begin{array}{c}
\includegraphics[scale=0.7]{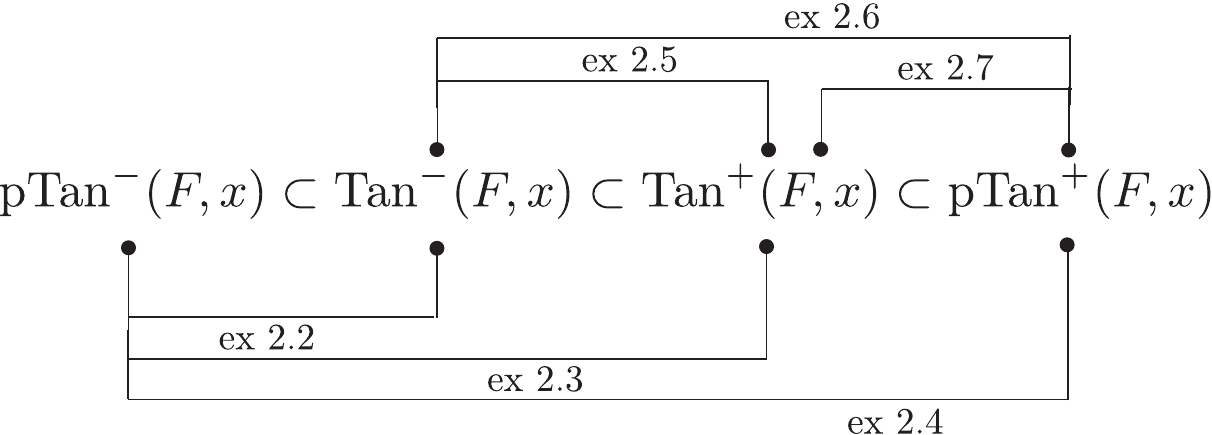}
\end{array}
\]
 \begin{example}
 $F:=\{0\}\cup\{1/m!:m\in \naturals_{1}\}$. 
 \emph{Here} $\pTan^{-}(F,0)=\Tan^{-}(F,0)=\{0\}$, $\Tan^{+}(F,0)=\reals_{+}$, $\pTan^{+}(F,0)=\reals$
 \end{example}
 
  \begin{example}
$F:=\reals_{+}$. \emph{Here} $\pTan^{-}(F,0)=\Tan^{-}(F,0)=\Tan^{+}(F,0)=\reals_{+}$, $\pTan^{+}(F,0)=\reals$.
 \end{example}
 
 \begin{example}
$F:=\{0\}$. \emph{Here} $\pTan^{-}(F,0)=\Tan^{-}(F,0)=\Tan^{+}(F,0)=\pTan^{+}(F,0)=\{0\}$.
 \end{example}

 \begin{example}
$F:=\{0\}\cup\{1/m:m\in\naturals_{1}\}$. \emph{Here} $\pTan^{-}(F,0)=\{0\}$, $\Tan^{-}(F,0)=\Tan^{+}(F,0)=\reals_{+}$, $\pTan^{+}(F,0)=\reals$.
 \end{example}

\begin{example}
  $F:=\{0\}\cup\{1/m:m\in\naturals_{1}\}\cup\{-1/m:m\in\naturals_{1}\}$. \emph{Here} $\pTan^{-}(F,0)=\{0\}$, $\Tan^{-}(F,0)=\Tan^{+}(F,0)=\pTan^{+}(F,0)=\reals$.
 \end{example}

 \begin{example}
 $F:=\{0\}\cup\{1/m!:m\in \naturals_{1}\}\cup\{-1/m:m\in \naturals_{1}\}$. \emph{Here} $\pTan^{-}(F,0)=\{0\}$, $\pTan^{-}(F,0)=-\reals_{+}$, $\Tan^{+}(F,0)=\pTan^{+}(F,0)=\reals$.
 \end{example}

\section{Tangency and paratangency in traditional sense compared with differentiability} \label{sec-tan-paratan-diff}

Traditionally, \emph{intrinsic notions} of tangent straight line to a curve and that of tangent plane to a surface at a point can  be resumed by the following general definition.

\begin{definition}\label{deftangtrad} Let $\hat x$ be an accumulation point of  a subset $F$ of $\reals^{n}$. A vector space $H$ of $\reals^{n}$ is said to be \emph{tangent in traditional sense} to $F$ at $\hat x$ if 
\begin{equation}
\lim\nolimits_{F\ni x\rightarrow \hat{x}}\dfrac{\dist(x,H+\hat x)}{%
\dist\left( x,\hat{x}\right) }=0.  \label{tangtrad}
\end{equation}%

\end{definition}

Since $\frac{\dist(x,H+\hat x)}{\dist\left( x,\hat{x}\right)}$ is the sinus of the angle between $H$ and the vector $x-\hat x$, the geometric meaning of (\ref{tangtrad}) is evident:  the half-line that passes through $\hat{x}$ and $x\in F$ and the affine space $H+\hat x$ form an angle that tends to zero as $x$ tends to $\hat{x}$.

\[
\begin{array}{cc}
\includegraphics[scale=0.5]{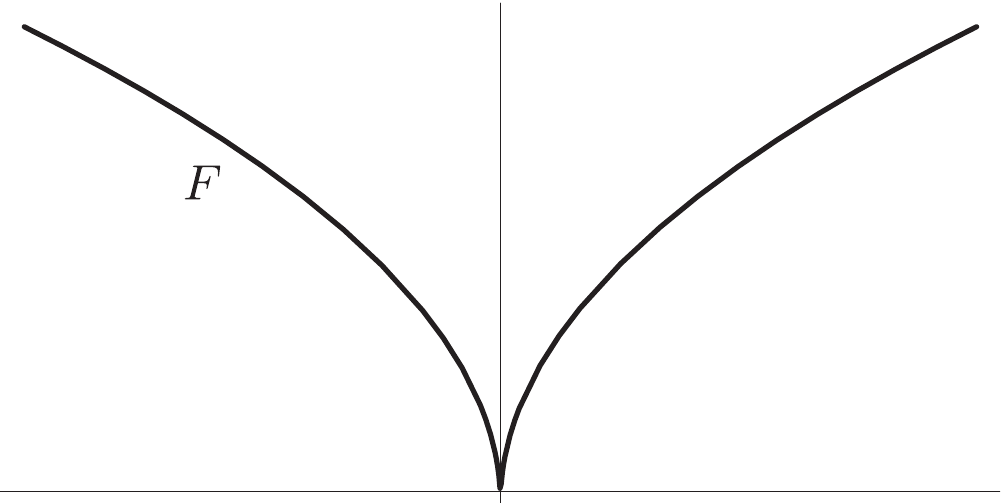}&
\includegraphics[scale=0.5]{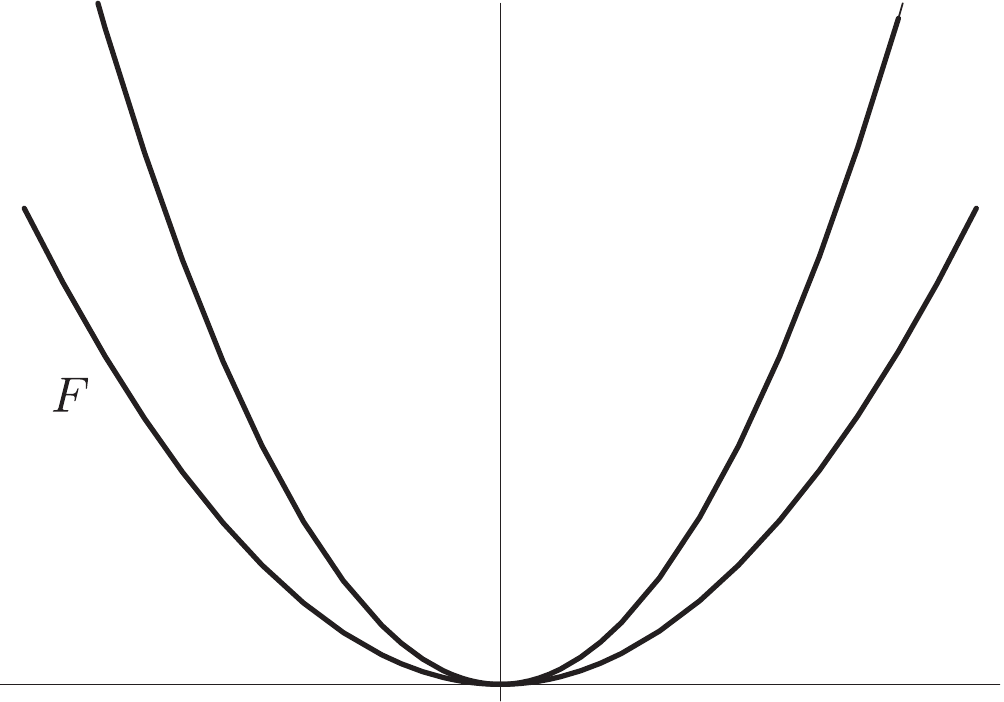}
\cr
\hbox{fig. 1: $F=\{(x,y):y^{3}=x^{2}\}$}&
\hbox{fig. 2:  $F=\{(x,y):(y-x^{2})(y-2x^{2})=0\}$}
\end{array}
\]
The sets $F$  of fig.\,$1$ and $2$ admit everywhere tangent line in traditional sense; in both cases  the tangent lines vary continuously. 

 Analytically Definition \ref{deftangtrad} becomes:
\begin{proposition}\label{proptangtrad-anal} A vector space $H$ of $\reals^{n}$ is   \emph{tangent in traditional sense} to $F$ at an accumulation point $\hat x$ of $F$ if and only if
\begin{equation}
\Tan^{+}(F,\hat{x})\subset H.\,\footnote{Something more, the upper tangent cone $\Tan^{+}(F,\hat{x})$ is the smallest closed set $H$ (not necessary, vector spaces) verifying \eqref{tangtrad}.} \label{tangtrad-anal}
\end{equation}%
\end{proposition}

According to \eqref{tangtrad-anal},  we assume, as a definition, that every vector space of $\reals^{n}$ is tangent in traditional sense to $F$ at the isolated points of $F$. 

\begin{proof}[\emph{\textbf{Proof}}]\emph{Necessity} of \eqref{tangtrad-anal}. Let $H$ be a vector space, tangent in traditional sense to $F$ at $\hat x$. Fix a non-null vector $v\in \Tan^{+}(F,\hat x)$. There exist  sequences $\{x_{m}\}_{m}\subset F$ and $\{\lambda_{m}\}_{m}\subset \reals_{++}$ such that $\lim_{m}\lambda_{m}= 0$, $\lim_{m}x_{m}=x$ and $\lim_{m}\frac{x_{m}-\hat x}{\lambda_{m}}=v $; hence $\lim_{m}\frac{x_{m}-\hat x}{\|x_{m}-\hat x\|}=\frac{v}{\|v\|} $. Therefore by \eqref{tangtrad} we have $0=\lim_{m}\frac{\dist(x_{m}, H+\hat x)}{\|x_{m}-\hat x\|}=\lim_{m}\dist\left(\frac{x_{m}-\hat x}{\|x_{m}-\hat x\|}, H\right)=\dist\left(\frac{v}{\|v\|},H\right)$. Then $v\in H$.

\emph{Sufficiency} of \eqref{tangtrad-anal}.  Suppose that \eqref{tangtrad} does not hold. Then there is a sequence $\{x_{m}\}_{m}\subset F\setminus\{\hat x\}$ with $\lim_{m}x_{m}=\hat x$ and 
\[
\lim_{m}\dist\left(\frac{x_{m}-\hat x}{\|x_{m}-\hat x\|},H\right)>0.
\leqno(*)\] 
By compactness, $\left\{\frac{x_{m}-\hat x}{\|x_{m}-\hat x\|}\right\}_{m}$  has a subsequence $\left\{\frac{x_{m_{k}}-\hat x}{\|x_{m_{k}}-\hat x\|}\right\}_{k}$ converging to a vector $w$ of norm $1$. Therefore, by \eqref{tangtrad-anal}, $w\in H$; consequently  $ \lim_{k}\dist\left(\frac{x_{m_{k}}-\hat x}{\|x_{m_{k}}-\hat x\|},H\right)= \dist(w,H)=0$, contradicting  $(*)$. 
\end{proof}

Following \textsc{Valiron} \cite[(1926)]{valiron_superficie}, \cite[(1927), p.\,47]{valiron_curva}, a vector space $H$ is said to be \emph{tangent in Valiron sense} to $F$ at $\hat x$, if $H=\Tan^{+}(F,\hat x)$.

Generally,   tangency was regarded as an elementary, intuitive notion not needing a definition. When definitions  however, were written down,  they were  neither precise and general, nor univocal and coherent. 
As a case study, the reader may consider the following  definitions of \textsc{Lagrange} and \textsc{Fréchet}.

\begin{enumerate}

\item\textsc{Lagrange} \cite[(1813), p.\,259]{lagrange1813} writes: 
 \begin{quotation}
 Ainsi, de même qu'une ligne droite peut être tangente d'une courbe, un plan peut être tangent d'une surface, et l'on déterminera le plan tangent par la condition qu'aucun autre plan ne puisse être mené par le point de contact entre celui-là et la surface.
\end{quotation}

\item\labelpag{fre1} \textsc{Fréchet} \cite[(1912), p.\,437]{Frechet_3}, \cite[(1964), p.\,189]{frechet64} writes: 
\begin{quotation}
Pr{\'e}cisons d'abord que nous entendons par plan tangent {\`a} [une surface] $S$ au point $(a,b,c)$ un plan qui soit le lieu des tangentes aux courbes situ{\'e}es sur $S$ et passant par ce point (s'entendant de celles de ces courbes qui ont effectivement une tangente en ce point).
\end{quotation}

\item\labelpag{fre2} \textsc{Fréchet}  \cite[(1964), p.\,193]{frechet64} writes: \begin{quotation}
Un plane $P$ passant par un point $Q$ d'une surface $S$ est, par définition, tangent à $S$ en $Q$ si,

1º) $M$ étant un point quelconque de $S$, distinct de $Q$, l'angle aigu de $M$ avec $P$ tend vers zéro quand $M$ tend vers $Q$,

2º) La condition $W$ ci-dessus [c'est-à-dire, si l'on projette $S$ sur $P$, il y a au moins un voisinage de $Q$ qui appartient entièrement à cette projection] est satisfaite. \footnote{In virtue of  \eqref{upperdist},  the definition \eqref{fre1} implies that \textsc{Fréchet}'s tangent plane to a surface is included in the corresponding upper tangent cone; while, surprisingly, by   Proposition \ref{proptangtrad-anal} the condition $1º)$ of the definition \eqref{fre2} demands the opposite set inclusion. Two examples: let $f,g:\reals^{2}\to\reals$ be defined by
\begin{equation}
f(x,y)=\begin{cases}
x&\text{if }x\in\{\frac{1}{m!}:m\in\naturals_{1}\} \text{ and }y=0\cr
0&\text{otherwise}
\end{cases},\quad 
g(x,y)=\begin{cases}
0&\text{if }x\ge0\cr
1&\text{otherwise.}
\end{cases}
\end{equation}
Observe that the plane $z=0$ is tangent to surface $z=f(x,y)$ at $(0,0,0)$ in the sense of the \textsc{Fréchet}'s definition \eqref{fre1}, but does not fulfill  \textsc{Fréchet}'s definition \eqref{fre2}. Conversely,  the plane $z=0$ is tangent to surface $z=g(x,y)$ at $(0,0,0)$ in the sense of the \textsc{Fréchet}'s definition \eqref{fre2}, but does not fulfill  \textsc{Fréchet}'s definition \eqref{fre1}. A critical attitude towards unstinting Fréchet's definition of tangency is took, for example, by \textsc{Valiron} \cite[(1926), p.\,191-192]{valiron_superficie} and \textsc{Cinquini} \cite[(1955)]{cinquini}.
}
\end{quotation}
\end{enumerate}

Following \textsc{Guareschi} \cite[(1934), p.\,175-176]{Guareschi-diff}, \cite[(1936), p.\,131-132]{Guareschi-Berzolari},
the vector space generated by $\Tan^{+}(F,\hat x)$ is called \emph{linear upper tangent space} of $S$ at $\hat x$; it is denoted by $\LTan^{+}(F,\hat x)$. Proposition \ref{proptangtrad-anal} allows us to describe  the linear upper tangent space as the smallest vector  space which is tangent in traditional sense.

In the case where $F$ is the graph of a function $f$, the tangency in traditional sense becomes differentiability.
\begin{proposition}[{\textsc{Guareschi} \cite[(1934), p.\,181, 183]{Guareschi-diff}, \cite[(1936), p.\,132]{Guareschi-Berzolari}\footnote{
\textsc{Guareschi} writes in \cite[(1936), p.\,131]{Guareschi-Berzolari}:
``Il concetto di semiretta tangente (o semitangente) ad un insieme puntuale in suo punto d'accumulazione [\dots], permette di strettamente collegare l'operazione analitica di differenziazione totale di una funzione di più variabili reali, in suo punto, al concetto sintetico di spazio lineare di dimensione minima contenente l'insieme tangente alla grafica della funzione stessa nel suo punto corrispondente; e tale collegamento porta a concludere che tanti sono i differenziali totali della funzione, quanti sono gli iperpiani passanti per quello spazio.''}, \textsc{severi} \cite[(1934), p.\,183-185]{Severi-diff}}]\label{proptangtrad-diff} If $A\subset\reals^{d}$,  $f:A\to\reals^{n}$ is a function,  $L:\reals^{d}\to\reals^{n}$ is  a linear function and  $\hat x\in A\cap\der(A)$, then the following three properties are equivalent:
 
\begin{enumerate}
\item\labelpag{diff1} $L$ is a differential of $f$  at $\hat x$, i.e., \begin{equation*}
\lim\nolimits_{A\ni x\rightarrow \hat{x}}\dfrac{f(x)-f(\hat{x})-L(x-\hat{x})%
}{\left\Vert x-\hat{x}\right\Vert }=0, \footnote{A function is said to be \emph{differentiable} at a point, when it admits a differential at the point.}
\end{equation*}%
\item\labelpag{diff2}  $f$ is continuous at $\hat x$ and $\graph(L)$ is tangent in traditional sense to $\graph(f)$ at $(\hat x,f(\hat x))$,
\item\labelpag{diff3}  $f$ is continuous at $\hat{x}$ and
$
\Tan^{+}(\graph (f),\left( \hat{x},f\left( \hat{x}%
\right) \right) )\subset \graph (L)
$.

\end{enumerate}

\end{proposition}

Contrary to the tangency of a straight line to a curve and that of a plane to a surface at a given  point,  there is no established tradition for the notion of paratangency. Nevertheless, to manifest  a logical  correlation between tangency and paratangency  we introduce the following definition.

\begin{definition}\label{defparatangtrad} Let $\hat x$ be an accumulation point of a subset $F$ of $\reals^{n}$. A vector space $H$ of $\reals^{n}$ is said to be \emph{paratangent in  traditional sense} to $F$ at $\hat x$ if 
\begin{equation}
\lim\nolimits_{F\ni x,y\rightarrow \hat{x}\atop y\neq x}\dfrac{\dist(x,H+ y)}{%
\dist\left( x,y\right) }=0.  \label{paratangtrad}
\end{equation}%

\end{definition}

Since $\frac{\dist(x,H+ y)}{\dist\left( x,y\right)}$ is the sinus of the angle between $H$ and  the vector $x-y$, the geometric meaning of (\ref{paratangtrad}) is evident:  the straight-lines passing through $y$ and $x$ in $F$ and the affine space $H+\hat x$  form  an angle that tends to zero as $x$ and $y$ tend to $\hat{x}$.

The concepts of tangency and paratangency were either  confused  or improperly identified. Surprisingly,  \textsc{Lebesgue} used paratangency to define traditional tangency in \emph{Du choix des définitions} \cite[(1934) p.\,6]{lebesgue-def}: 
  \begin{quotation}
  L'idée de tangent provient de ce jugement: tout arc suffisement petit d'une courbe est rectiligne. [\dots] Dire qu'un élément de courbe et un élément de droite sont indiscernables, c'est dire que la droite est pratiquement confondue avec toute corde de l'élément de courbe, d'où cette définition: a) \emph{on dit qu'une courbe $C$ a des tangentes si, quel que soit le point $M$ de $C$ et de quelque manière qu'on fasse tendre les points $M_{1}$ et $M_{2}$ de $C$ vers $M$, la corde $M_{1}M_{2}$ tend vers une position limite déterminée par la donnée de $M$; on l'appelle tangente en $M$.}
  \end{quotation}
Paratangency to curves and surfaces was used by \textsc{Peano}  in \emph{Applicazioni geometriche} \cite[(1887), p.\,163, 181-184]{peano87} to evaluate the infinitesimal quotient of the length of an arc and its segment or its projection. Moreover,  the following proposition due to \textsc{Peano}   makes transparent the relation between
paratangency and $C^{1}$-smoothness.  

\begin{proposition}
[{\textsc{Peano} \cite[(1887), teorema II, p.\,59]{peano87}}]\label{prop:Ppara} If $\gamma$ is a continuously
differentiable curve and $\gamma'(\hat{t})\neq 0$, then its tangent straight
line at $\gamma[\hat t]$ is the limit of the lines passing through $\gamma[t]$ and $ \gamma[u]$    when $t\ne u$ and $t,u$ tend to $\hat{t}$.
\end{proposition}

 Analytically we have that
\begin{proposition}\label{propparatangtrad-anal}  A vector space $H$ of $\reals^{n}$ is   \emph{paratangent in  traditional sense} to $F$ at an accumulation point $\hat x$ of $F$ if and only if
\begin{equation}
\pTan^{+}(F,\hat{x})\subset H.\,   \label{paratangtrad-anal}
\end{equation}%
\end{proposition}

\begin{proof}[\emph{\textbf{Proof}}] Proceed similarly as in the  proof of Proposition \ref{proptangtrad-anal}. \footnote{\label{foot1}Among the non-empty sets $H$ (not necessary, vector spaces) verifying \eqref{paratangtrad}, the smallest closed set is the upper paratangent cone.}
\end{proof}

According to \eqref{paratangtrad-anal},  we assume, as a definition, that every vector space of $\reals^{n}$ is paratangent in traditional sense to $F$ at the isolated points of $F$. 

As for upper tangent cones,  the \emph{linear upper paratangent space} of $F$ at $\hat x$, i.e.,  the linear hull of the upper paratangent cone of $F$ at $\hat x$, was  introduced by  \textsc{Guareschi} \cite[(1941), p.\,154]{Guareschi-Roma}; it is denoted by $\pLTan^{+}(F,\hat x)$.  By Proposition \ref{propparatangtrad-anal}, the linear upper paratangent space is the smallest vector space which is paratangent in traditional sense.

In the case where $F$ is a graph of a function $f$, the paratangency in traditional sense becomes strict differentiability.

\begin{definition}[{\textsc{Peano} in \cite[(1892)]{peano_mathesis} for $n=d=1$, 
\textsc{Severi} in \cite[(1934), p.\,185]{Severi-diff}}] Let $A\subset\reals^{d}$, let $f:A\to\reals^{n}$ be a function and let $\hat x\in A\cap\der(A)$. The function $f$ is said to be \emph{strictly differentiable at} $\hat x$, if there is a linear function $L:\reals^{d}\to\reals^{n}$ (called \emph{strict differential} of $f$ at $\hat x$) such that 
\begin{equation*}
\lim\nolimits_{A\ni x,y\rightarrow \hat{x}\atop y\ne x}\dfrac{f(x)-f(y)-L(x-y)%
}{\left\Vert x-y\right\Vert }=0.\,\footnote{As usual, ``strictly differentiable function on a set $X$'' stands for ``strictly differentiable function at every point belonging to $X$''. 
}
\end{equation*}

\end{definition}

\begin{proposition}[{\textsc{Severi} \protect\cite[(1934), p.\,189]{Severi-diff}, \textsc{Guareschi} \protect\cite[(1941), p.\,161]{Guareschi-Roma}}]\label{propparatangtrad-diff} If $A\subset\reals^{d}$,  $f:A\to\reals^{n}$ is a function,  $L:\reals^{d}\to\reals^{n}$ is  a linear function and  $\hat x\in A\cap\der(A)$, then  the following three properties are equivalent:
\begin{enumerate}
\item\labelpag{sdiff1}  $L$ is a strict differential of $f$  at $\hat x$,
\item\labelpag{sdiff2}  $f$ is continuous at $\hat x$ and $\graph(L)$ is paratangent in traditional sense to $\graph(f)$ at $(\hat x,f(\hat x))$,
\item\labelpag{sdiff3}  $f$ is continuous at $\hat{x}$ and
$
\pTan^{+}(\graph (f),\left( \hat{x},f\left( \hat{x}%
\right) \right) )\subset \graph (L)
$.

\end{enumerate}

\end{proposition}

In order to prove this Proposition \ref{propparatangtrad-diff} we use the following two lemmata.

\begin{lemma}
[\textbf{Cyrenian lemma for strict differentiability}]\label{cireneoC1} Let $f$, $L$, $A$ and  $\hat x$ be as  Proposition  $\ref{propparatangtrad-diff}$. Then $L$ is a strict
differential of $f$ at $\hat{x}$ if and only if 
\begin{enumerate}
\item $\lim\nolimits_{m}%
\dfrac{f(x_{m})-f(y_{m})}{\lambda _{m}}=L(v)$ for each $v\in \mathbb{R}%
^{d} $ and for all sequences $\left\{ \lambda _{m}\right\} _{m}\subset \mathbb{R}_{++}$, $\left\{ x_{m}\right\} _{m}, \left\{ y_{m}\right\} _{m}\subset A$ such that 
$%
\lim\nolimits_{m}\lambda _{m}=0$, $\lim_{m\to\infty}y_{m}=\hat x$ and
$\lim\nolimits_{m}\dfrac{x_{m}-y_{m}}{\lambda _{m}}=v$.\hfill$\Box$\footnote{Immediate consequences of Cyrenian lemma  \ref{cireneoC1} are the chain rule for strictly differentiable functions and the following inverse function theorem: ``Let $A$ be a non empty subset of $\reals^{d}$ such that $A\subset\der(A)$. If $f:A\to\reals^{n}$ is injective, strictly differentiable and non-singular on $A$, then  $(a)$ $(f_{|A})^{-1}$ is strictly differentiable on $f(A)$ and $(b)$ $f(A)$ is a $C^{1}$-manifold, whenever $A$ is a $C^{1}$-manifold''. }
\end{enumerate}

\end{lemma}

\begin{lemma}
\label{lipsch}  Let $A$ be a subset of $\reals^{d}$ and let $\hat x$. A function $f:A\to\reals^{n}$ is locally Lipschitz at $\hat x$ if and only if
 $f$ is continuous at $\hat x$ and $\pTan^{+}(\graph(f), (\hat x,f(\hat x)))$ does not contain vertical lines. \footnote{We will say that a cone  $B\subset \reals^{d}\times\reals^{n}$  does \emph{not contain vertical lines}, whenever  do not exist  pairs $(v,w)\in\reals^{d}\times\reals^{n}$  such that $(v,w)\in B$, $v=0$ and $w\ne0$.} \hfill$\Box$
\end{lemma}

\begin{proof}[\emph{\textbf{Proof of Proposition \ref{propparatangtrad-diff}}}]  $\eqref{sdiff2}\Longleftrightarrow\eqref{sdiff3}$: It follows from Proposition \ref{propparatangtrad-anal}.

$\eqref{sdiff1}\Longrightarrow\eqref{sdiff3}$: Let $v\in \pTan^{+}\big(\graph(f),(\hat x, f(\hat x)\big)$. Take $v_{1}\in\reals^{d}$ and $v_{2}\in\reals^{n}$ such that $v=(v_{1},v_{2})$. By definition of $\pTan^{+}$,   there are $\{\lambda_{m}\}_{m}\subset\reals_{++}$, $\{x_{m}\}_{m}, \{y_{m}\}_{m}\subset A$ such that $\lim_{m} \lambda_{m}=0$, $\lim_{m}y_{m}=\hat x$   and $\lim_{m}\frac{x_{m} -y_{m}}{\lambda_{m}}=v_{1}$ and  $\lim_{m}\frac{f(x_{m})-f(y_{m})}{\lambda_{m}}=v_{2}$. Hence, by Cyrenian lemma \ref{cireneoC1} one has $L(v_{1})=v_{2}$, that is $v=(v_{1},v_{2})\in\graph(L)$. 

$\eqref{sdiff3}\Longrightarrow\eqref{sdiff1}$. By Lemma \ref{lipsch}, condition  \eqref{sdiff3} implies the existence of $M, \varepsilon, \in\reals_{++}$ such that
\[\|f(x)-f(y)\|\le M\|x-y\|\text{ for every }x,y\in \B_{\varepsilon}(\hat x).
\leqno(*)\]

To prove \eqref{sdiff1}, by absurd suppose that \eqref{sdiff1} does not hold. Then there are  sequences $\{x_{m}\}_{m}, \{y_{m}\}_{m} \subset A$ with $x_{m}\ne y_{m}$ for $m\in\naturals$ such that $\lim_{m\to\infty}y_{m}=\hat x$ and 
\[
\lim_{m\to\infty}\left\|\frac{f(x_{m})-f(y_{m})}{\|x_{m}-y_{m}\|}-L(\frac{x_{m}-y_{m}}{\|x_{m}-y_{m}\|})\right\|>0.
\leqno(**)\]
By compactness and $(*)$, without loss of generality we assume that 
$\lim_{m}\frac{x_{m}-y_{m}}{\|x_{m}-y_{m}\|}=v_{1}\in\reals^{d}$ and $\lim_{m}\frac{f(x_{m})-f(y_{m})}{\|x_{m}-y_{m}\|}=v_{2}\in\reals^{d}$. Hence $(v_{1},v_{2})\in\pTan^{+}(\graph (f),\left( \hat{x},f\left( \hat{x}%
\right) \right) )$. By \eqref{sdiff3} we have that $(v_{1},v_{2})\in\graph(L)$, that is $L(v_{1})=v_{2}$. Therefore \\ $\lim_{m}\left\|\frac{f(x_{m})-f(y_{m})}{\|x_{m}-y_{m}\|}-L(\frac{x_{m}-y_{m}}{\|x_{m}-y_{m}\|})\right\|=\|v_{2}-L(v_{1})\|=0$, contradicting $(**)$.
\end{proof}

\begin{corollary}\label{corosdiff} Let $A$ be a subset of $\reals^{d}$ and let $\hat x\in  A\cap\der(A)$. A function $\varphi:A\to\reals^{n}$ is  strictly differentiable at $\hat x$  if and only if $\varphi$ is continuous at $\hat x$ and  $\pLTan(\graph(\varphi), (\hat x,\varphi(\hat x)))$ does not contain vertical lines.

\end{corollary}

\begin{proof}[\emph{\textbf{Proof}}] 
\emph{Necessity}. Assume $\varphi$ strictly differentiable at $\hat x$ and denote by $L$ a strict differential of $\varphi$ at $x$.  By Proposition \ref{propparatangtrad-diff} the function $\varphi$ is continuous at $x$ and 
 \[
\pLTan^{+}(\graph (\varphi),\left(\hat x,\varphi\left( \hat x%
\right) \right) )\subset \graph(L).
\leqno(*)\] 
Therefore, $\pLTan^{+}(\graph (\varphi),\left( \hat x,\varphi\left( \hat x%
\right) \right) )$ does not contain verticals lines. 

 \emph{Sufficiency}. Posit $d':=\dim(\pLTan^{+}(\graph (\varphi),\left( \hat x,\varphi\left( \hat x
\right) \right) ))$,  $R_{n}:=\{(v,w)\in\reals^{d}\times\reals^{n}: v=0\}$ and $V:=(R_{n}\oplus \pLTan^{+}(\graph (\varphi),\left( \hat x,\varphi\left( \hat x
\right) \right) ))^{\perp}$ in $\reals^{d}\times\reals^{n}$. 
Since  $\pLTan^{+}(\graph (\varphi),\left( \hat x,\varphi\left( \hat x%
\right) \right) )$ does not contain verticals lines,  the vector space \[
\Lambda:=V\oplus\, \pLTan^{+}(\graph (\varphi),\left( \hat x,\varphi\left(\hat  x%
\right) \right) ) \] has dimension $d$ and it does not contain verticals lines; hence, there  exists a linear function $L:\reals^{d}\to \reals^{n}$ such that $\graph (L)=\Lambda$. Therefore, being $\varphi$ continuous at $\hat x$, by Proposition \ref{propparatangtrad-diff} we have that $\varphi$ is strictly differentiable at $\hat x$ and  $L$ is a strict differential of $\varphi$ at $\hat x$.
\end{proof}

\section{Grassmann Exterior algebra, limits of vector spaces and angles between  vector spaces}\label{sec-grassmann}

Let $\Lambda(\reals^{n})$ denote the graded Grassmann exterior algebra on $n$-dimensional Euclidean space $\reals^{n}$:
\begin{equation}
\Lambda(\reals^{n}):=\Lambda_{0}(\reals^{n})\oplus\Lambda_{1}(\reals^{n})\oplus\Lambda_{2}(\reals^{n})\oplus\Lambda_{3}(\reals^{n})\oplus\cdots\oplus\Lambda_{n}(\reals^{n})
\end{equation}
where $\Lambda_{0}(\reals^{n}):=\reals$ and $\Lambda_{k}(\reals^{n})$, $k=1,\dots,n$, is the $\binom n k$-dimensional vector space generated by linear combinations of simple $k$-vectors  $\land_{i=1}^{k}v_{i}$, where $\{v_{i}\}_{i=1}^{k}\subset\reals^{n}$. 

Euclidean inner product of $\reals^{n}$ induces an inner product on  $\Lambda(\reals^{n})$;   on simple $k$-vectors it is described by
\begin{equation}
\langle\land_{i=1}^{k}v_{i},\land_{i=1}^{k}w_{i}\rangle:=\det\big(\langle v_{i},w_{j}\rangle\big)_{i j}.
\end{equation}
With respect to the associated norm on $\Lambda(\reals^{n})$, a   simple $k$-vectors  $\land_{i=1}^{k}v_{i}$ has a non-null norm if and only if the $k$ vectors $\{v_{i}\}_{i=1}^{k}$ are linearly independent.

Let $\GG(\reals^{n},d)$ denote the set of all $d$-dimensional  vector (sub)spaces of $\reals^{n}$ and define $\GG(\reals^{n}):=\bigcup_{0\le d\le n}\GG(\reals^{n},d)$. 
If $d\ge1$, the \emph{angle} between two (non oriented) $d$-dimensional vector spaces $V$ and $W$ is a real number denoted by $\ang(V,W)$ and  well defined by
\begin{equation}\ang(V,W):=\arccos\left(\frac{|\langle\land_{i=1}^{d}v_{i},\land_{i=1}^{d}w_{i}\rangle|}{\|\land_{i=1}^{d}v_{i}\|\,\|\land_{i=1}^{d}w_{i}\|}\right).\label{angle}
\end{equation}
where $\{v_{i}\}_{i=1}^{d}$ and $\{w_{i}\}_{i=1}^{d}$ are arbitrary 
bases of $V$ and $W$, respectively.

 For every basis $\{v_{i}\}_{i=1}^{d}$ of a $d$-dimensional vector space $V$, one has:  
\begin{enumerate}
\item\labelpag{pre1grass}the norm $\|\land_{i=1}^{d}v_{i}\|$ is the $d$-dimensional elementary measure  of $d$-parallelepiped $\Para(\{v_{i}\}_{i=1}^{d}):=\{\sum_{i=1}^{d}\alpha_{i}v_{i}:\sum_{i=1}^{d}\alpha_{i}=1 \text{ and }0\le\alpha_{i}\le1\text{ for } 1\le i\le d\}$; 
\item\labelpag{pre3grass} $\dist(x,V)=\frac{\|(\land_{i=1}^{d}v_{i})\land x\|}{\|\land_{i=1}^{d}v_{i}\|}$ for every $x\in\reals^{n}$;

\item\labelpag{pre2grass} if $v_{i}'$ is the orthogonal projection of vector $v_{i}$ on a $d$-dimensional vector space $W$, then \[\|\land_{i=1}^{d}v_{i}'\|=\|\land_{i=1}^{d}v_{i}\|\cos(\ang(V,W)),\] 
i.e., $\cos(\ang(V,W))$ is the reduction factor 
for $d$-dimensional measure under orthogonal projection of $V$ on $W$.\,\footnote{To verify \eqref{pre2grass}, fix an orthonormal basis $\{w_{i}\}_{i=1}^{d}$ of $W$. Then observe that $v_{i}'=\sum_{k=1}^{d}\langle v_{i},w_{k}\rangle w_{k}$ for $1\le i\le d$, and $\cos(\ang(V,W))=\frac{|\langle\land_{i=1}^{d}v_{i},\land_{i=1}^{d}w_{i}\rangle|}{\|\land_{i=1}^{d}v_{i}\|}$. Hence $\|\land_{i=1}^{d}v_{i}'\|^{2}=\det\big(\langle\sum_{k=1}^{d}\langle v_{i},w_{k}\rangle w_{k}, \sum_{k=1}^{d}\langle v_{j},w_{k}\rangle w_{k}\rangle\big)_{ij}=\det\big(\sum_{k=1}^{d}\langle v_{i},w_{k}\rangle \langle v_{j},w_{k}\rangle\big)_{ij}=\det^{2}\big(\langle v_{i},w_{j}\rangle\big)_{ij}=\langle\land_{i=1}^{d}v_{i},\land_{i=1}^{d}w_{i}\rangle^{2}=\|\land_{i=1}^{d} v_{i}\|^{2}\cos^{2}(\ang(V,W))$.
}
 Hence $\ang(V,W)=\frac{\pi}{2}$ if and only if $V\cap W^{\bot}\ne\{0\}$. 

\end{enumerate}

\begin{proposition}[see {\textsc{Peano} \cite[(1887) thm.\,5 p.\,35, thm.\,7 p.\,39]{peano87} and \cite[(1887) thm.\,2 p.\,34, thm.\,2 p.\,36]{peano87}}]\label{propconv} If $V_{m}, V\in\GG(\reals^{n},d)$ and $d\ge1$, then  following properties are equivalent:
\begin{enumerate}
\item\labelpag{tre1} $V\subset \Li_{m\to\infty}V_{m}$,
\item\labelpag{tre2} $\lim_{m\to\infty}\ang(V_{m},V)=0$,
\item\labelpag{tre3} there are bases  $\{v_{i}\}_{i=1}^{d}$ and $\{v_{i}^{(m)}\}_{i=1}^{d}$ of $V$ and  $V_{m}$ respectively, such that $v_{i}=\lim_{m\to\infty}v_{i}^{(m)}$ for  $1\le i\le d$.
\end{enumerate}
\end{proposition}

\begin{proof}[\emph{\textbf{Proof}}] 

$\eqref{tre1}\Longrightarrow \eqref{tre3}$. Let $\{v_{i}\}_{i=1}^{d}$ be a basis of $V$. By \eqref{tre1}, for every $1\le i\le d$ there is a sequence $\{v_{i}^{(m)}\}_{m}$ such that $v_{i}^{(m)}\in V_{m}$ and $\lim_{m}v_{i}^{(m)}=v_{i}$. Hence, by continuity of exterior product, one obtains that 
\[
\lim_{m}\land_{i=1}^{d} v_{i}^{m}=\land_{i=1}^{d} v_{i}.
\leqno(*)\] 
Being $\{v_{i}\}_{i=1}^{d}$ a basis of $V$, one has $\|\land_{i=1}^{d} v_{i}\|\ne0$; therefore, $(*)$ implies that $\{v_{i}^{m}\}_{i=1}^{d}$ is a basis of $V_{m}$, since  $\|\land_{i=1}^{d} v_{i}^{(m)}\|\ne0$ eventually in $m$. 

$\eqref{tre3}\Longrightarrow\eqref{tre2}$.  Let $\{v_{i}\}_{i=1}^{d}$ and $\{v_{i}^{(m)}\}_{i=1}^{d}$ as in $\eqref{tre3}$.  By continuity of exterior product we have $(*)$.
Therefore, from Definition \eqref{angle} the required \eqref{tre2} follows.

$\eqref{tre2}\Longrightarrow\eqref{tre1}$. Let $\{v_{i}\}_{i=1}^{d}$ be an orthonormal basis of $V$; moreover, for every $m\in\naturals$, let $\{v_{i}^{(m)}\}_{i=1}^{d}$ be an orthonormal basis of $V_{m}$ such that $\langle \land_{i=1}^{d}v_{i}^{(m)}, \land_{i=1}^{d}v_{i}\rangle\ge0$. From orthonormality it follows that $\|\land_{i=1}^{d}v_{i}\|=\|\land_{i=1}^{d}v_{i}^{m}\|=1$; hence, by condition  \eqref{tre2}, we have that $\lim_{m\to\infty}\land_{i=1}^{d}v_{i}^{m}=\land_{i=1}^{d}v_{i}$. On the other hand, by continuity of exterior product, one obtains that
\[
\lim_{m}\dist(x,V_{m})=\lim_{m}\|(\land_{i=1}^{d}v_{i}^{m})\land x\|=\|(\land_{i=1}^{d}v_{i})\land x\|=\dist(x,V)
\]
for every $x\in\reals^{n}$. Therefore, for every $x\in V$, we have $
\lim_{m}\dist(x,V_{m})=0$, i.e., $x\in \Li_{m}V_{m}$.
\end{proof}

\begin{corollary}\label{coroconv1} Let $V_{m}, V$ be as in Proposition \ref{propconv}.  The following equivalence holds: 
\begin{equation}
 V\subset\Li_{m\to\infty}V_{m}\iff \dist(x,V)=\lim_{m\to\infty}\dist(x,V_{m}) \text{ for all } x\in\reals^{n}.
\label{tre4}
\end{equation}

\end{corollary}

\begin{proof}[\emph{\textbf{Proof}}] $\stackrel{\eqref{tre4}}\Longleftarrow$: By the definition of the lower limit  $\Li$, it is obvious. $\stackrel{\eqref{tre4}}\Longrightarrow$: Let $\{v_{i}\}_{i=1}^{d}$ be a basis of $V$. By the definition of the lower limit of sets,  $V\subset\Li_{m\to\infty}V_{m}$ implies that for every $1\le i\le d$ there is a sequence $\{v_{i}^{(m)}\}_{m}$ such that $v_{i}^{(m)}\in V_{m}$ and $\lim_{m}v_{i}^{(m)}=v_{i}$. Hence, by the continuity of the exterior product and $\|\land_{i=1}^{d}v_{i}\|\ne0$, one has that 
\[
\lim_{m}\dist(x,V_{m})=\lim_{m}\frac{\|(\land_{i=1}^{d}v_{i}^{m})\land x\|}{\|\land_{i=1}^{d}v_{i}^{m}\|}=\frac{\|(\land_{i=1}^{d}v_{i})\land x\|}{\|\land_{i=1}^{d}v_{i}\|}=\dist(x,V)
\leqno(*)\]
for every $x\in\reals^{n}$.

\end{proof}

\begin{corollary}\label{coroconv2} Let $V_{m}, V$ be as in  Proposition \ref{propconv}. The following equivalence holds:
\begin{equation}
 V\subset\Li_{m\to\infty}V_{m}\iff
\Ls_{m\to\infty}V_{m}\subset V.
\label{tre6}
\end{equation}

\end{corollary}

\begin{proof}[\emph{\textbf{Proof}}] $\stackrel{\eqref{tre6}}\Longrightarrow$: Being $V\subset \Li_{m}V_{m}$, by \eqref{tre4} we have that 
\[
\lim_{m}\dist(x,V_{m})=\dist(x,V)
\leqno(*)\]
for every $x\in\reals^{n}$. Now, for $x\in \Ls_{m}V_{m}$, Definition of upper limit of sets entails $\liminf_{m}\dist(x,V_{m})=0$;  therefore, by $(*)$ one obtains $\dist(x, V)=0$, i.e., $x\in V$.

$\stackrel{\eqref{tre6}}\Longleftarrow$: Fix $\hat x\in V$ and choose an orthonormal basis $\{v_{i}^{(m)}\}_{i=1}^{d}$ of $V_{m}$. By absurd,  suppose $\hat x\not\in\Li_{m}V_{m}$, i.e., $\limsup_{m}\dist(\hat x,V_{m})=\alpha>0$. Then there exist $\{w_{i}\}_{i=1}^{d}\subset\reals^{n}$ and  an infinite subset  $N$ of $\naturals$ such that $\lim_{N\ni m\to\infty}\dist(\hat x,V_{m})=\alpha$ and $\lim_{N\ni m\to\infty}v_{i}^{m}=w_{i}$ for $1\le i\le d$. Now, let $W$ denote the $d$-dimensional vector space generated by the orthonormal basis $\{w_{i}\}_{i=1}^{d}$. First, by the continuity of exterior product we have
\[
\lim_{N\ni m\to\infty}\dist(\hat x,V_{m})=\lim_{N\ni m\to\infty}\|(\land_{i=1}^{d}v_{i}^{m})\land \hat x\|=\|(\land_{i=1}^{d}w_{i})\land \hat x\|=\dist(\hat x,W);
\]
hence $\hat x\not\in W$, since $\lim_{N\ni m\to\infty}\dist(\hat x,V_{m})=\alpha>0$. On the other hand, from `$\eqref{tre3}\Longrightarrow \eqref{tre1}$' of Proposition \ref{propconv} it follows that $W\subset \Li_{N\ni m\to\infty}V_{m}$; hence 
$W\subset \Li_{N\ni m\to\infty}V_{m}\subset \Ls_{m\to\infty}V_{m}\subset V
$; therefore, being both $W$ and $V$ d-dimensional vector spaces, we have $W=V$ in conflict with the fact that  $\hat x\in V\setminus  W$.
\end{proof}

Let us define continuity of set-valued functions. Let $A\subset \reals^{n}$ and $\hat x\in A$; as usual, a set-valued function $\varphi: A\to\calP(\reals^{n})$ is said to be \emph{lower} (resp. \emph{upper}) \emph{semicontinuous at} $\hat x$, whenever $\varphi(\hat x)\subset\Li_{A\ni x\to\hat{x}}\varphi(x)$ \footnote{i.e., the set $\{x\in A: B\cap\varphi( x)\ne\emptyset\}$ is open in $A$, for every open ball $B$.} (resp.\,$\Ls_{A\ni x\to\hat{x}}\varphi(x)\subset \varphi(\hat x)$ \footnote{i.e., the set $\{x\in A: \overline{B}\cap\varphi( x)\ne\emptyset\}$ is closed in $A$, for every closed ball $\overline{B}$.}); moreover $\varphi$ is said to be \emph{continuous} at $\hat x$, if $\varphi$ is both lower and upper semicontinuous. \footnote{Since $\Li_{A\ni x\to\hat{x}}\varphi(x)\subset \Ls_{x\ni A\to\hat{x}}\varphi(x)$, the continuity of $\varphi$ at $\hat x$ amounts to the equalities: $\varphi(\hat x)=\Li_{A\ni x\to\hat{x}}\varphi(x)=\Ls_{A\ni x\to\hat{x}}\varphi(x)$.}
The following well known elementary properties are useful:

\begin{enumerate}
\item\label{lower_seq} $\varphi$ is lower semicontinuous at $\hat{x}$ if and only if
\[\varphi(\hat x)\subset\Li_{m\to\infty}\varphi(x_{m}) \text{ for every sequence }\{x_{m}\}_{m}\subset A\text{ converging to }\hat{x};
\]
\item\label{upper_seq} $\varphi$ is upper semicontinuous at $\hat{x}$ if and only if
\[\Ls_{m\to\infty}\varphi(x_{m})\subset \varphi(\hat{x}) \text{ for every sequence }\{x_{m}\}_{m}\subset A\text{ converging to }\hat{x};\]
\item\label{cont_seq} $\varphi$ is continuous at $\hat{x}$ if and only if
\[\Ls_{m\to\infty}\varphi(x_{m})\subset \varphi(\hat{x})\subset\Li_{m\to\infty}\varphi(x_{m}) \text{ for every }\{x_{m}\}_{m}\subset A\text{ converging to }\hat{x}.
\]

\end{enumerate}

By \eqref{lower_seq}-\eqref{cont_seq}, Proposition \ref{propconv}  and Corollary \ref{coroconv2} can be restated in terms of continuity of   vector-space-valued functions.

\begin{theorem}\label{theoconv} Let $A$ be a subset of $\reals^{n}$ and let $\tau:A\to\GG(\reals^{n},d)$ be a vector-space-valued function with $d\ge1$. For every $x\in A$, the following properties are equivalent:
\begin{enumerate}
\item $\tau$ is lower semicontinuous at $x$, \item  $\tau$ is  upper semicontinuous at $x$,
\item   $\tau$ is continuous at $x$,

\item $\lim_{A\ni y\to x}\ang(\tau(y),\tau(x))=0$. \hfill$\Box$
\end{enumerate}
 
\end{theorem}

\section{Four-cones Coincidence Theorem: local and global version}\label{sec-main}

The geometry of manifolds was originated by \textsc{Riemann}'s Habilitationsschrift (1854) \emph{\"Uber die Hypothesen welche der Geometrie zu Grunde liegen} \cite[(1878) p.\, 272-287]{riemann}.
 Today axiomatic formulation of manifolds by coordinate systems and regular atlas  was presented  by \textsc{Veblen} and \textsc{Whitehead} in \cite[(1931)]{veblen31,veblen}  
 and was elaborated in a series of celebrated works by  \textsc{Whitney} in the 1930's. There are various kinds of finite dimensional manifolds; all are  topological, i.e., they are locally homeomorphic to Euclidean spaces.  In the sequel we will consider (sub)manifolds of Euclidean spaces that are $C^{1}$ smooth.

\begin{definition}\label{defsot} A  non-empty subset  $S$ of $\reals^{n}$  is said to be a $C^{1}$-\emph{manifold of $\reals^{n}$ at a point  }$ x\in S$, if there are an open neighborhood $\Omega$ of $x$ in $\reals^{n}$, an affine subspace $H$ of $\reals^{n}$ and a $C^{1}$-diffeomorphism  from $\Omega$ onto another open set of $\reals^{n}$ such that 
\begin{equation}
\xi(S\cap\Omega)=\xi(\Omega)\cap H.\label{eqsot}
\end{equation}
 The dimension of $H$ is said to be the \emph{dimension of $S$ at $x$} and it is denoted by $\dim(S,x)$. 
 \end{definition}

 A set $S$ is said to be  a $C^{1}$-\emph{manifold of }$\reals^{n}$, if it is  a $C^{1}$-manifold of $\reals^{n}$ at every point.  
 If 
 the dimension $d=\dim(S,x)$ does  not depend on $x\in S$, then $S$ is said to be  a \emph{$d$-dimensional $C^{1}$-manifold}. \footnote{If $S$ is  a $C^{1}$-manifold of $\reals^{n}$, then $S_{i}:=\{x\in S: \dim(S,x)=i\}$ is an $i$-dimensional $C^{1}$-manifold (if non-empty), and, for $0\le i\ne j\le n$, $S_{i}$ and $S_{j}$ are \emph{separated} (i.e. $S_{i}\cap\overline{S_{j}}=S_{j}\cap\overline{S_{i}}=\emptyset$).}

 Elementary and well known facts on arbitrary $C^{1}$-manifolds are resumed in the following proposition.

\begin{proposition}\label{proppre} Let $S$ be a $C^{1}$-manifold of $\reals^{n}$ and let $x$ be an arbitrary point of $S$. Then the following four properties hold:
\begin{enumerate}
\item\labelpag{pre0} $S$ is a topological manifold (hence, a locally compact set);

\item\labelpag{pre2}  $\Tan^{+}(S,x)$ is a vector space 
and
\[\dd\xi(x)\big(\Tan^{+}(S,x)\big)=H-\xi(x)  \text{ and } \dim(S,x)=\dim(\Tan^{+}(S,x))
\]
where $\xi$ and $H$ are as in Definition $\ref{defsot}$  and $\dd\xi(x)$ denotes the differential of $\xi$ at $x$; 
\item\labelpag{pre3}  all four tangent cones coincide, i.e.
\[\pTan^{+}(S,x)=\Tan^{+}(S,x)=\Tan^{-}(S,x)=\pTan^{-}(S,x);\,
\footnote{The relevant equality $\Tan^{+}(S,x)=\pTan^{-}(S,x)$ was proved by \textsc{Clarke} \cite[(1975), pp.\,254-256]{clarke75} for $C^{1}$-manifolds and convex sets.}
\]

\item\labelpag{pre4} the tangent vector space $\Tan^{+}(S,x)$ varies  continuously, i.e., the map $z\mapsto \Tan^{+}(S,z)$ is continuous on $S$, or, equivalently, for every $z\in S$
 \[\quad\qquad\qquad\Li_{S\ni y\to z}\Tan^{+}(S,y)\subset \Tan^{+}(S,z)\subset\Ls_{S\ni y\to z}\Tan^{+}(S,y). \quad \qquad\qquad\qquad\Box
 \]
\end{enumerate}
\end{proposition}

\begin{proof}[\emph{\textbf{Proof}}]

 All the four discussed cones  share the following three \emph{basic  properties}. Let $\Tan$ denote any one of the four tangent cones $\Tan^{+}$,  $\Tan^{-}$,  $\pTan^{+}$ and  $\pTan^{-}$, then 
\begin{enumerate}
\item\labelpag{localstable} they are \emph{local}, i.e., 
\[\Tan(S\cap \Omega, x)= \Tan(S, x)\] 
for every  neighborhood $\Omega$ of $x$ in $\reals^{n}$;
\item\labelpag{diffstable} they are \emph{stable by diffeomorphisms}, i.e., 
\[\Tan(\xi(S\cap\Omega), \xi(x))=\dd\xi(x)( \Tan(S\cap\Omega, x)) \]
 for every open neighborhood $\Omega$ of $x$ and for every $C^{1}$-diffeomorphism from $\Omega$ to another open set of $\reals^{n}$; \footnote{The  equality in  $\eqref{diffstable}$ is a consequence of the description  of $\Tan$ in terms of sequences (see \eqref{upperseq}, \eqref{lowerseq}, \eqref{paraseq} and \eqref{clarkeseq}), since $\xi$ is $C^{1}$-diffeomorphism and, consequently, by Cyrenian lemma \ref{cireneoC1} one has 
\[\lim_{m\to\infty}\frac{x_{m}-y_{m}}{\lambda_{m}}=v \iff \lim_{m\to\infty}\frac{\xi(x_{m})-\xi(y_{m})}{\lambda_{m}}=\dd\xi(x)(v)
\leqno(*)\]
for $\lambda_{m}\to0^{+}$, $v\in\reals^{n}$, $\{y_{m}\}_{m},\{x_{m}\}_{m} \subset S\cap\Omega$ with $y_{m}\to x$.}

\item\labelpag{vectstable} they  \emph{fix vector spaces}, i.e., 
\[\Tan(V, 0)= V\]
for every  vector subspace $V$ of $\reals^{n}$.
\end{enumerate}
To prove both \eqref{pre2} and \eqref{pre3}, fix $x\in S$ and let $(\Omega, H, \xi)$ 
be as in Definition \ref{defsot}.  The following equalities hold: 
\[
\Tan(\xi(S\cap\Omega), \xi(x))=\dd\xi(x)\big(\Tan(S\cap\Omega,x)\big)=\dd\xi(x)\big(\Tan(S,x)\big).
\leqno(*)\]
The first and the second equality of $(*)$ are due to  $\eqref{diffstable}$ and $\eqref{localstable}$, respectively.   On the other hand
\[
\Tan(\xi(\Omega)\cap H,\xi(x))=\Tan(H, \xi(x))=\Tan(H-\xi(x),0 )=H-\xi(x),
\leqno(**)\]
where the equalities are due to $\eqref{localstable}$, $\eqref{diffstable}$ and $\eqref{vectstable}$, respectively.
 Finally, combining $(*)$ and $(**)$ with \eqref{eqsot}, we have \eqref{pre2} and \eqref{pre3}. Besides, the tangent cone $\Tan^{+}(S,x)$ is a vector subspace of $\reals^{n}$ having the same dimension of $S$ at $x$, because it is the preimage of the vector space $H-\xi(x)$ under the linear isomorphism $\dd\xi(x)$.
%
%
%
%
%
%
%
%

To verify \eqref{pre4}, fix $x\in S$ and let $(\Omega, H, \xi)$ 
be as in Definition \ref{defsot}. From \eqref{pre2} it follows that  
\[\dd\xi(y)\big(\Tan^{+}(S,y)\big)=H-\xi(y)=H-\xi(x)\text{ for every }y\in S\cap\Omega.
\leqno(***)\]
 \emph{$1^{st}$ case: $\dim(S,x)=0$}.  Obviously $\eqref{pre4}$ holds, since   $x$ is an isolated point of $S$.    \emph{$2^{st}$ case: $d:=\dim(S,x)\ge1$}. Choose a basis $\{w_{i}\}_{i=1}^{d}$ of the vector space $W:=H-\xi(x)$ and a sequence $\{y_{m}\}_{m}\subset S$ converging to $x$. Since $\Omega$ is an open neighborhood of $x$,  there is a natural number $\bar m$ such that $y_{m}\in S\cap\Omega$ for every natural number $m\ge\bar m$. For every $y_{m}$ with $m\ge \bar m$, by $(***)$ define a basis $\{v_{i}^{(m)}\}_{i=1}^{d}$ of the $d$-dimensional vector space $\Tan^{+}(S,y_{m})$ by
\[v_{i}^{m}:=\big(\dd\xi(y_{m})\big)^{-1}(w_{i})\text{ for }1\le i\le d
\]
and, analogously,  define a basis $\{v_{i}\}_{i=1}^{d}$ of the $d$-dimensional vector space $\Tan^{+}(S,x)$ by
\[v_{i}:=\big(\dd\xi(x)\big)^{-1}(w_{i})\text{ for }1\le i\le d
\]
 Since $\xi$ is $C^{1}$- diffeomorphism, the map $y\mapsto\big(\dd\xi(y)\big)^{-1}$ is continuous. Hence 
 \[\lim_{m\to\infty}v_{i}^{m}=v_{i} \text{ for }1\le i\le d.
 \]
 Therefore, by Proposition \ref{propconv}, $\Tan^{+}(S,x)\subset \Li_{m\to\infty}\Tan^{+}(S,y_{m})$; consequently, by Corollary \ref{coroconv2}, $\Li_{m\to\infty}\Tan^{+}(S,y_{m})\subset\Tan^{+}(S,x)$. Hence Theorem \ref{theoconv} entails \eqref{pre4}. \end{proof}

\begin{lemma}\label{corosdiffsequel} Let $F$ be a  subset of $\reals^{n}$ such that $F\subset\der(F)$ and 
\begin{equation}
e_{n}\not\in\pLTan^{+}(F, x))\text{ for every }x\in F.
 \label{corosdiffseq1}
\end{equation}
Let $\hat x\in F$. Then  there exist $\varepsilon\in\reals_{++}$, $A\subset\reals^{n-1}$ with $ A\subset\der(A)$  and  there exists a  function $\varphi:A\to\reals$ strictly differentiable on $A$ such that 
\begin{equation}
\label{corodiffseq2}
\graph(\varphi)=F\cap \B_{\varepsilon}(\hat x).
\end{equation} 
Moreover, if $F$ is a $d$-dimensional topological manifold (resp. locally compact at $\hat x$),  then  $A$ is a $d$-dimensional topological manifold (resp. locally compact at $\hat t$, where $\hat t$ is the element of $A$ such $\hat{x}=(\hat{t}, \varphi(\hat{t}))$).  \end{lemma}

\begin{proof}[\textbf{\emph{Proof}}] By \eqref{coll7} there are $\varepsilon\in\reals_{++}$, $A\subset\reals^{n-1}$ and $\varphi:A\to\reals$ such that 
\[\varphi \text{ is continuous and }
\graph(\varphi)=F\cap \B_{\varepsilon}(\hat x).
\leqno(*1)\]
Since  $\B_{\varepsilon}(\hat x)$ is an open set, we have
 \[
 \pLTan^{+}(\graph(\varphi),(t,\varphi(t)))=\pLTan^{+}(F,(t,\varphi(t))\text{  for every }t\in A;
 \leqno(*2)\]
therefore, by \eqref{corosdiffseq1}
\[ \pLTan^{+}(\graph(\varphi),(t,\varphi(t)))\text{ does not contain vertical lines}
 \leqno(*3)\]
for every point $t\in A':=\{t\in A: e_{n}\not \in \pLTan^{+}(F, (t,\varphi(t)))$.   On the other hand, being $F\subset \der(F)$,   $\graph(\varphi)$ has no isolated point; therefore 
 $A\subset\der(A)$.  Now, by Corollary \ref{corosdiff} we have that $\varphi$ is strictly differentiable on $A'$, as required. Finally, if $F$ is locally compact at $\hat{x}$ (resp. a $d$-dimensional topological manifold), then 
   the set $A$ is locally compact at $\hat{t}$ (resp. a $d$-dimensional topological manifold), since it is homeomorphic to  $\graph(\varphi)=F\cap \B_{\varepsilon}(\hat x)$.
\end{proof}

 \begin{lemma}\label{lemUno-strict}
Let $A\subset\reals^{d}$ with $A\subset\der(A)$ and let $\varphi:A\to\reals^{n}$ be strictly differentiable.  For every $x\in A$, the following properties hold:
\begin{enumerate}
\item\labelpag{lem1} $\pTan^{+}\big(\graph(\varphi),(x,\varphi(x))\big)=\{(v,L(v)): v\in \pTan^{+}(A,x)\}$,

\item\labelpag{lem2} $\pTan^{-}\big(\graph(\varphi),(x,\varphi(x))\big)=\{(v,L(v)): v\in \pTan^{-}(A,x)\}$,

\end{enumerate}
where $L$ denote  a strict  differential of $\varphi$ at $x$. 

\end{lemma}

\begin{proof}[\emph{\textbf{Proof}}] Fix $x\in A$. We will prove  only \eqref{lem1};  a similar proof of the \eqref{lem2} 
is left to the reader.
\emph{$1^{st}$ claim}: $\pTan^{+}\big(\graph(\varphi),(x,\varphi(x))\big)\subset\{v,L(v)): v\in \pTan^{+}(A,x)\}$. 
Let $v\in\reals^{d}$ and $r\in\reals^{n}$ such that $(v,r)\in  \pTan^{+}(\graph(\varphi),(x,\varphi(x)))$. Then there exist sequences $\{\lambda_{m}\}_{m}\subset\reals_{++}$ and  $\{x_{m}\}_{m}, \{y_{m}\}_{m}\subset A$ such that $\lim_{m}\lambda_{m}=0$, $\lim_{m}(x_{m},\varphi(x_{m}))=( x,\varphi(x))$ and 
\[
\lim_{m\to\infty}\frac{(x_{m}-y_{m},\varphi(x_{m})-\varphi(y_{m}))}{\lambda_{m}}=(v,r).
\leqno(*1)\]
 By  Cyrenian lemma \ref{cireneoC1}, we have $\lim_{m}\frac{\varphi(x_{m})-\varphi(y_{m})}{\lambda_{m}}=L(v)$, since $\lim_{m}\frac{x_{m}-y_{m}}{\lambda_{m}}=v$. Hence, $v\in \pTan^{+}(A,x)$ and $(v,r)=(v,L(v))$, as required.

 \emph{$2^{rd}$ claim}: $\{(v,L(v)): v\in \pTan^{+}(A,x)\}\subset \pTan^{+}\big(\graph(\varphi),(x,\varphi(x))\big)$.
  Let $v\in \pTan^{+}(A,x)$. Then there exist  sequences $\{\lambda_{m}\}_{m}\subset\reals_{++}$, $\{x_{m}\}_{m}, \{y_{m}\}_{m}\subset A$ such that $\lim_{m}\lambda_{m}=0$, $\lim_{m\to\infty}x_{m}=x$ and
\[\lim_{m\to\infty}\frac{x_{m}-y_{m}}{\lambda_{m}}=v.
\leqno(*2)\]
Being $\varphi$ strictly differentiable, property $(*2)$ and Cyrenian lemma \ref{cireneoC1} imply 
\[\lim_{m\to\infty}\frac{\varphi(x_{m})-\varphi(y_{m})}{\lambda_{m}}=L(v).
\leqno(*3)\]
Since $\lim_{m}(x_{m},\varphi(x_{m}))=(x,\varphi(x))$ and $\lim_{m}\lambda_{m}=0$, from $(*2)$ and $(*3)$ follows that 
$(v,L(v))\in \pTan^{+}\big(\graph(\varphi),(x,\varphi(x))\big)$, as required.

 \end{proof}

\begin{lemma}\label{lemDue-strict}
Let $A\subset\reals^{n}$ with $A\subset\der(A)$ and let $\varphi:A\to\reals$ be strictly differentiable. If  $A$ is a $C^{1}$-manifold of $\reals^{n}$ at a point $\hat t\in A$, then $\graph(\varphi)$ is a $C^{1}$-manifold of $\reals^{n+1}$ at $(\hat t,\varphi(\hat t))$.
\end{lemma}

\begin{proof}[\emph{\textbf{Proof}}] \emph{$1^{st}$ case: $A$ is a non-empty open subset of some vector subspace $V$ of $\reals^{n}$}. Let $V^{\perp}$ denote the orthogonal complement of $V$. Moreover, let $\nu: (A+V^{\perp})\times\reals\to (A+ V^{\perp})\times\reals$ the function such that 
 $\nu(t+y,z):=(t+y, z-\varphi(t))$ for every $(t,y,z)\in A\times V^{\perp}\times\reals$. 
Since both domain and codomain of $\nu$ coincide with the open set $\Omega:=(A+ V^{\perp})\times\reals$ of $\reals^{n+1}$ and, besides, $\nu$ is bjective and strictly differentiable, we have that $\nu$ is a $C^{1}$-diffeomorphism from $\Omega$ onto $\Omega$. On the other hand we have 
\[\nu\big(\graph(\varphi)\cap \Omega\big)=\nu\big(\graph(\varphi)\big)=A\times\{0\}=\Omega\cap (V\times\{0\})=\nu(\Omega)\cap (V\times\{0\}).
\leqno(*1)\]
Therefore, by Definition \ref{defsot}, $\graph(\varphi)$ is a $C^{1}$-manifold of $\reals^{n+1}$.

\emph{$2^{nd}$ case: $A$ is an arbitrary non-empty subset of  $\reals^{n}$.}
 Since $A$ is a $C^{1}$-manifold at $\hat t$, there exist an open neighborhood $\Omega$ of $\hat t$ in $\reals^{n}$, a vector subspace $V$ of $\reals^{n}$ and a $C^{1}$-diffeomorphism $\xi$ from $\Omega$ to another open subset of $\reals^{n}$ such that
\[\xi(A\cap\Omega)=\xi(\Omega)\cap V.
\leqno(*2)\]
Now, let us define  the strict differentiable function $\psi:\xi(\Omega)\cap V\to\reals$ by $\psi(y):= \varphi(\xi^{-1}(y))$. Since the domain of $\psi$ is an open subset of the vector space $V$, by the first  case we have that $\graph(\psi)$ is a $C^{1}$-manifold of $\reals^{n+1}$. On the other hand, let us define the $C^{1}$-diffeomorphism  $\mu:\Omega\times\reals\to \xi(\Omega)\times\reals$ by $\mu(t,r):=(\xi(t),r)$.
Clearly, 
\[\mu^{-1}(\graph(\psi))=\graph(\varphi)\cap\big((A\cap\Omega)\times\reals\big)=\graph(\varphi)\cap(\Omega\times\reals).
\leqno(*3)\]
Hence, 
\[\graph(\varphi)\cap(\Omega\times\reals)\text{ is a $C^{1}$-manifold of }\reals^{n+1},
\leqno(*4)\]
 since it is the image of the $C^{1}$-manifold $\graph(\psi)$ by the $C^{1}$-diffeomorphism $\mu^{-1}$. Therefore $\graph(\varphi)$ is a $C^{1}$-manifold at $(\hat t,\varphi(\hat t))$, because $\Omega\times\reals$ is an open neighborhood of $(\hat t, \varphi(\hat t))$.
\end{proof}

\begin{theorem}[\textbf{Four-cones coincidence theorem: local version}]\label{theomainlem}
Let $F\subset\reals^{n}$ and let $\hat x\in F$. Then $F$  is  a $C^{1}$-manifold at $\hat x$ if and only if the following three properties hold:
\begin{enumerate}
\item \labelpag{mainlablem0}
 $F$ is locally compact at $\hat x$,
\item \labelpag{mainlablem1}
$\pTan^{-}(F,\hat x)=\pTan^{+}(F,\hat x)$,
\item \labelpag{mainlablem2} there exists an open ball $\B_{\delta}(\hat x)$ centered at $\hat x$ such that  \[\pTan^{+}(F,x)=\pLTan^{+}(F,x)
\] for  every $x\in F\cap \B_{\delta}(\hat x)$.  
\end{enumerate}

\end{theorem}

\begin{proof}[\emph{\textbf{Proof}}] \emph{Necessity}. Let $F$ be a $C^{1}$-manifold of $\reals^{n}$ at $\hat x$. Clearly, $F$ is locally compact at $\hat x$. On the other hand, Property \eqref{pre3} of Proposition \ref{proppre} implies the coincidence of the lower and upper paratangent cones, as required. \emph{Sufficiency}. Assume \eqref{mainlablem0}, \eqref{mainlablem1} and \eqref{mainlablem2} are true.

\emph{$1^{st}$ case: $\dim(\pLTan^{+}(F,\hat x))=\{0\}$}. Since $\Tan^{+}(F,\hat x)\subset\pLTan^{+}(F,\hat x)$, we have $\Tan^{+}(F,\hat x)=\{0\}$. Hence, by \eqref{coll6}, $\hat x$ is an isolated point of $F$. Therefore $F$ is a $C^{1}$-manifold of dimension zero at $\hat x$.

\emph{$2^{nd}$ case: $\dim(\pLTan^{+}(F,\hat x))=n$}. By  $\eqref{mainlablem1}$ we have $\pTan^{-}(F,\hat x)=\reals^{n}$. Hence,  being $F$ locally compact at $\hat x$, property  \eqref{coll5rock} implies $\hat x\in\intt(F)$. Therefore $F$ is a $C^{1}$-manifold of dimension $n$ at $\hat x$, since $\hat x$ is an interior point of $F$.

\emph{$3^{rd}$ case: $0<\dim(\pLTan^{+}(F,\hat x))<n$}. Choose  two non null vectors $v_{0}, v_{n}\in\reals^{n}$ such that $v_{0}\in \pLTan^{+}(F,\hat x))$ and $v_{n}\not\in \pLTan^{+}(F,\hat x))$. Without loss of generality, assume that $w=e_{n}$. By \eqref{coll4} we have
\[\pTan^{-}(F,\hat x)=\Li_{F\ni x\to \hat x}\Tan^{+}(F,x).
\leqno(*1)\]
Moreover, by \eqref{mainlablem1} and \eqref{mainlablem2}, the equality $\pTan^{-}(F,\hat x)= \pLTan^{+}(F,\hat x)$ holds; hence $(*1)$ implies $v_{0}\in \Li_{F\ni x\to \hat x}\Tan^{+}(F,x)$; consequently, there exists a positive real number $\delta_{1}\le\delta$ such that the non null vector $v_{0}\in \Tan^{+}(F,x)$ for every $x\in F\cap\B_{\delta_{1}}(\hat x)$. Therefore, by \eqref{coll6}, 
\[F\cap\B_{\delta_{1}}(\hat x)\subset\der(F\cap\B_{\delta_{1}}(\hat x))
\leqno(*2)\] 
Now, by \eqref{coll2} and \eqref{mainlablem2} we have 
\[\Ls_{F\ni x\to\hat x}\pLTan^{+}(F,x)\subset \pLTan^{+}(F,\hat x);
 \leqno(*3)\]
 consequently, $e_{n}\not\in \Ls_{F\ni x\to\hat x}\pLTan^{+}(F,x)$. Therefore, there exists a positive real number $\delta_{2}\le\delta_{1}$ such that
 \[e_{n}\not\in \pLTan^{+}(F,x)\text{ for every }x\in F\cap \B_{\delta_{2}}(\hat x).
 \leqno(*4)\]
 From $(*2)$, $(*4)$ and  
  Lemma \ref{corosdiffsequel}
  there are a positive real number $\varepsilon\le \delta_{2}$, a subset $A$ of $\reals^{n-1}$ with $A\subset\der(A)$ and there exists a function $\varphi:A\to\reals$ strictly differentiable such 
\[\graph(\varphi)=F\cap \B_{\varepsilon}(\hat x).
\leqno(*5)\]
and  
\[A \text{ is locally compact at }\hat t
\leqno(*6)\]
  where $\hat t\in A$ and $(\hat{t},\varphi(\hat t))=\hat x$.  Being $\B_{\varepsilon}(\hat x)$ an open set, by $(*5)$ and $(5.6)$ we have $\pTan^{-}(\graph(\varphi),(t,\varphi(t)))=\pTan^{-}(F,(t,\varphi(t)))$ and $\pTan^{+}(\graph(\varphi),(t,\varphi(t)))=\pTan^{+}(F,(t,\varphi(t)))$  for every $t\in A$. Hence, by \eqref{mainlablem1} and \eqref{mainlablem2} we obtain  
\[\pTan^{-}(\graph(\varphi),(\hat t,\varphi(\hat t)))=\pTan^{+}(\graph(\varphi),(\hat t,\varphi(\hat t)))
\leqno(*7)\]
\and 
\[\pTan^{+}(\graph(\varphi),( t,\varphi( t)))=\pLTan^{+}(\graph(\varphi),( t,\varphi( t)))\text{ for every }t\in A.
\leqno(*8)\]

Therefore, from Lemma \ref{lemUno-strict} it follows that
\[\pTan^{-}(A,\hat t)=\pTan^{+}(A,\hat t)
 \leqno(*7')\]
and 
 \[\pTan^{+}(A,t)=\pLTan^{+}(A, t) \text{ for every }t\in A.
 \leqno(*8')\]
Now, by induction, assume that this theorem \ref{theomainlem} holds for subsets of $\reals^{n-1}$. Then, by $(*6)$, $(*7')$ and $(*8')$ , we have that the subset $A$ of $\reals^{n-1}$ is a $C^{1}$-manifold at $\hat t$. Therefore, by Lemma \ref{lemDue-strict},   $\graph(\varphi)$ (i.e. $F\cap\B_{\varepsilon}(\hat x)$) is a $C^{1}$-manifold of $\reals^{n}$ at $\hat x$, as required.
\end{proof}

\begin{theorem}[\textbf{Four-cones coincidence theorem: global version}]\label{theomain}
A  non-empty subset $F$ of $\reals^{n}$ is  a $C^{1}$-manifold  if and only if $F$ is locally compact and  
the lower and upper paratangent cones to $F$ coincide at every point, i.e.,
\begin{equation}\label{mainlab1}
\pTan^{-}(F,x)=\pTan^{+}(F,x)\text{ for every }x\in F.
\end{equation}

\end{theorem}

It is certainly worthwhile to remark that, by \eqref{4cones} and \eqref{coll4},  condition \eqref{mainlab1} amounts to the set inclusion
\begin{equation}\label{mainlab2}
\pTan^+(F, x)\subset \Li_{F\ni y\to  x}\Tan^+( F, y)\text{ for every }x\in F.
\end{equation}

\begin{proof}[\emph{\textbf{Proof}}] \emph{Necessity}. Let $F$ be a $C^{1}$-manifold of $\reals^{n}$. Clearly, $F$ is locally compact. On the other hand, Property \eqref{pre3} of Proposition \ref{proppre} implies the coincidence of the lower and upper paratangent cones, as required. \emph{Sufficiency}. Let $F$ be  locally compact and let both lower and upper paratangent cones coincide at every point.  Then  \eqref{coll1} and \eqref{coll3} imply that the upper paratangent cones to $F$ are  vector space, i.e.,
\[\pTan^{+}(F,x)=\pLTan^{+}(F,x)
\leqno(*1)\] 
for every $x\in F$. Hence, Theorem \ref{theomainlem} implies that $F$ is $C^{1}$-manifold of $\reals^{n}$ at its every point, as required.
\end{proof}

Let us denote with $\GL_n(\R)$ the general linear group, i.e. the multiplicative group of the $n\times n$ invertible matrices with real entries. We denote with $E$ the unit of $\GL_n(\R)$. Let $M_{n}(\R)$ be the algebra of $n\times n$ matrices, endowed with Euclidean topology. 
Clearly every subgroup of $\GL_n(\R)$ which is a $C^1$ manifold of  $M_{n}(\R)$, is necessarily a locally compact set with respect  Euclidean topology. Conversely, we will apply main Theorem \ref{theomain} to prove that 

\begin{corollary}[\textsc{von Neumann}, {\cite[(1929)]{vonN}}]\label{coroneumann} A locally compact subgroup $\calG$ of $\GL_n(\R)$ is a $C^1$-manifold  of  $M_{n}(\R)$.
\end{corollary}

\begin{proof}[\textbf{\emph{Proof}}] By Theorem \ref{theomain}, it is enough to prove \eqref{mainlab2}, that is, we must prove that
\[\pTan^+(\calG,A)\subset \Li_{H\to A}\Tan^+(\calG,H)
\leqno(*1)\] 
for every $A\in\calG$. 
\emph{$1^{st}$ case: Let $A$ be the unit $E$.}
 Let $V\in \pTan^+(\calG,E)$; by definition there exist three sequences $\{\lambda_m\}_m\subset\reals_{++}$,
$\{A_m\}_m\subset\calG$,
$\{B_m\}_m\subset\calG$ such  that 
$$\lim_m\lambda_m= 0^+,\qquad\lim_mA_m= E,\qquad \lim_m\frac{A_m-B_m}{\lambda_m}=V.\leqno(*2)$$
In order to show  that $V\in\Li_{H\to E}\Tan^+(\calG,H)$,  
let $\{H_k\}_k\subset\calG$ be such that $\lim_kH_k=E$. Define sequences of matrices
$\{H_{k,m}\}_m$ and $\{V_k\}_k$ 
by
$$H_{k,m}:=H_k\cdot A_m\cdot B^{-1}_m, \qquad V_k:=H_k\cdot V.
\leqno(*3)$$
Observe that  
\[V_k\in\Tan^+(\calG,H_k) \text{ for every }k\in\mathbb N,
\leqno(*4)\] since 
$\{H_{k,m}\}_m\subset\mathcal{G}$, $\lim_mH_{k,m}=H_k$ and 
$\lim_m\frac{H_{k,m}-H_k}{\lambda_m}=\lim_m\frac{ H_k\cdot A_m\cdot B^{- 1}_m-H_k}{\lambda_m}=\lim_mH_{k}\cdot\frac{(A_m-B_m)}{\lambda_m}\cdot B^{-1}_m=H_k\cdot V\cdot E=V_k.$ On the other hand, $\lim_kV_k= V$. Hence, because $\{H_k\}_k\subset\calG$ is an arbitrary sequence converging to $E$, from the definition of the lower limit of sets it follows that $V\in \Li_{H\to E}\Tan^+(\calG,H)$.

\emph{$2^{nd}$ case: $A$ is an arbitrary element of  $\calG$.} Since $\frac{A_{m}-B_{m}}{\lambda_{m}}=A\frac{A^{-1}A_{m}-A^{-1}B_{m}}{\lambda_{m}}$ for every $A_{m},B_{m}, \in\calG$ and $\lambda_{m}\in\reals_{++}$, by the definitions of tangent and upper paratangent cones we have
\[\Tan^+(\calG,A)=A\cdot\Tan^+(\calG,E)\text{ and }\pTan^+(\calG,A)=A\cdot\pTan^+(\calG,E).
\leqno(*5)\]
 Therefore, from the \emph{$1^{st}$ case} it follows that 
  $\pTan^+(\calG,A)=A\cdot\pTan^+(\calG,E)\subset A\cdot\Li_{H\to E}\Tan^+(\calG,H)=\Li_{H\to A}\Tan^+(\calG,H)$, as required.
\end{proof}

Continuous variability of tangent spaces (in traditional sense)  does not assure that a set is a $C^{1}$-manifold (for example, see fig.\,$1$  and $2$ of Section \ref{sec-tan-paratan-diff}). 
In order to  characterize $C^{1}$-manifolds, in the following two corollaries simple conditions are added  to the continuous variability of tangent spaces. 
\[
\begin{array}{cc}
\includegraphics[scale=0.5]{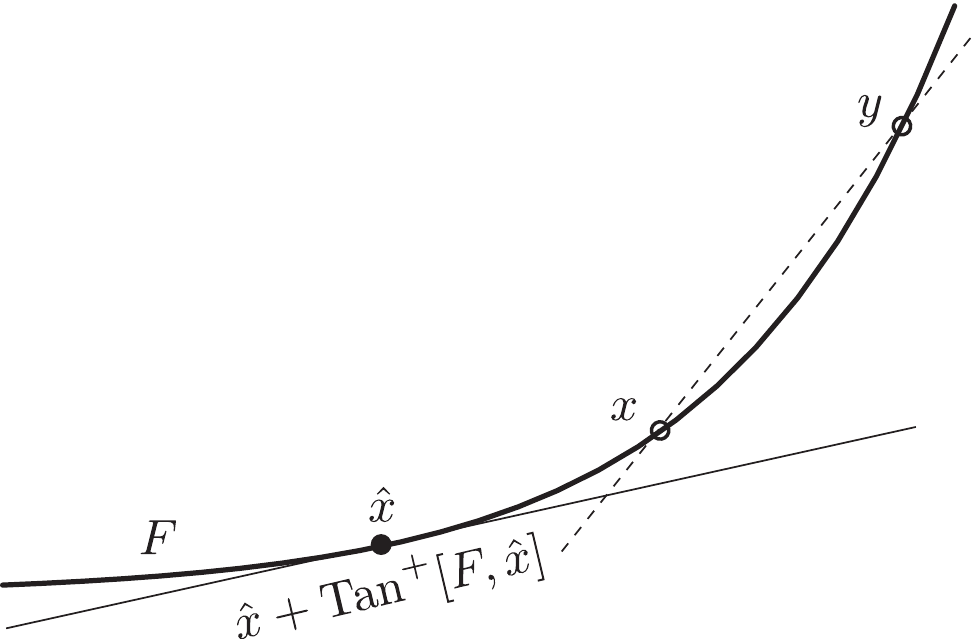}&
\includegraphics[scale=0.5]{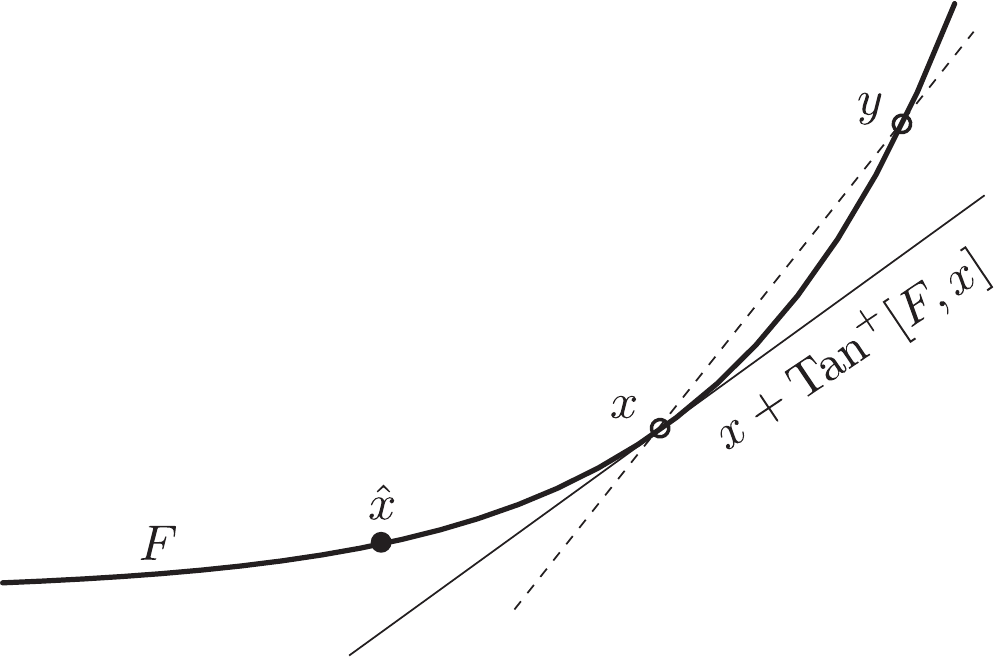}
\cr
\hbox{fig. 3}&
\hbox{fig. 4}
\end{array}
\]

\begin{corollary}\label{coro3A}
A  non-empty subset  $F$ of $\reals^{n}$ is  a  $C^{1}$-manifold  if and only if $F$ is locally compact and the following two properties hold. 
\begin{enumerate}
\item \labelpag{cont1} the map $x\mapsto\Tan^{+}(F, x)$ is lower semicontinuous on $F$,
\item \labelpag{cont1piu}$\lim_{F\ni x,y\to \hat{x}\atop y\neq x}\frac{\dist(y,x+\Tan^{+}(F,\hat{x}))}{\|y-x\|}=0$ for every $\hat x\in\der(F)$.

\end{enumerate}
\end{corollary}

In the case where $\Tan^{+}(F,\hat{x})$ is a vector space, condition \eqref{cont1piu} means that the angle between the straight-line passing through two distinct points $y$ and $x$ of $F$ and the vector space $\Tan^{+}(F,\hat{x})$ tangent to $F$ at $\hat x$, tends to zero as $x$ and $y$ tend to $\hat{x}$ (see fig.\,3 above).

\begin{proof}[\textbf{\emph{Proof}}] By \eqref{coll4} the lower semicontinuity of $x\mapsto\Tan^{+}(F,x)$ amounts to 
\[
\Tan^{+}(F,x)=\pTan^{-}(F,x)\text{ for every }x\in F
\leqno(*)
\]
On the other hand, by footnote \ref{foot1}, condition \eqref{cont1piu} becomes
\[\pTan^{+}(F,\hat{x})\subset \Tan^{+}(F)\text{ for every }\hat{x}\in F.
\leqno(**)\]
Hence, from conditions \eqref{cont1} and \eqref{cont1piu} it follows that $\pTan^{+}(F,\hat{x})\subset \pTan^{-}(F,x)$; and conversely. Therefore Theorem \ref{theomain} entails both necessity and sufficiency of the conditions \eqref{cont1} and \eqref{cont1piu}.
 \end{proof}

\begin{corollary}\label{coro3AA}
A  non-empty subset  $F$ of $\reals^{n}$ is  a  $C^{1}$-manifold  if and only if $F$ is locally compact and the following two properties hold:  
\begin{enumerate}
\item \labelpag{cont2} 
 the map $x\mapsto\Tan^{+}(F, x)$ is continuous on $F$,
\item \labelpag{cont2piu}$\lim_{F\ni x,y\rightarrow \hat{x}\atop y\neq x}\frac{\dist(y,x+\Tan^{+}(F,x))}{\|y-x\|}=0$  for every $\hat x\in\der(F)$.
\end{enumerate}

\end{corollary}

In the case where the upper tangent cones $\Tan^{+}(F, x)$ are vector spaces, condition \eqref{cont2piu} means that the angle between the straight-line passing through two distinct points $y$ and $x$ of $F$ and the tangent vector space to $F$ at $x$   tends to zero as $x$ and $y$ tend to $\hat{x}$ (see fig.\,4 above).

\begin{proof}[\textbf{\emph{Proof}}] \emph{Necessity}: it is known. \emph{Sufficiency}.
 We must prove \eqref{mainlab2}. Hence, it is enough to show that, for every $\hat x\in \der (F)$,   the following set inclusion holds
\begin{equation}\label{concoro3A}\pTan^+(F,\hat x)\cap\{v\in\reals^{n}: \|v\|=1\}\subset \Li_{F\ni y\to \hat x}\Tan^{+}(F,y).
\end{equation}
 To prove \eqref{concoro3A} fix $\hat x\in \der(F)$ and $v\in \pTan^+(F,\hat x)$ with $\|v\|=1$. By \eqref{paraseq} there exist sequences $\{x_{m}\},\{y_{m}\}_{m}\subset F$ converging to $\hat x$ such that $\lim_{m\to\infty}\frac{y_{m}-x_{m}}{\|y_{m}-x_{m}\|}=v$. By  \eqref{cont2piu} and the following triangular inequality
\[\dist(v,\Tan^{+}(F,x_{m}))\le \frac{\dist(y_{m},x_{m}+\Tan^{+}(F,x_{m}))}{\|y_{m}-x_{m}\|}+\|v-\frac{y_{m}-x_{m}}{y_{m}-x_{m}}\|,
\]
we have that $\lim_{m\to\infty}\dist(v,\Tan^{+}(F,x_{m}))=0$; hence,  $v\in\Ls_{F\ni x\to\hat{x}}\Tan^{+}(F,x)$. Thus, from \eqref{cont2} it follows that $v\in \Li_{F\ni y\to \hat x}\Tan^{+}(F,y)$, as \eqref{concoro3A} requires. \footnote{In the proof of Corollary we have show that condition \eqref{cont2piu} imply the following set inclusion: ``$\pTan^{+}(F,\hat{x})\subset\Ls_{x\to\hat x}\Tan^{+}(F,x)$ for every $\hat{x}\in F$''. The converse also holds.}
\end{proof}

Both old and modern  characterization of $C^{1}$-manifolds can be deduced from  four-cones concidence theorem \ref{theomain};  as example, we state and prove the following theorem, due to  \textsc{Tierno} (see {\cite[(1997)]{tiernoA}, \cite[(2000)]{tiernoB}}).

\begin{theorem} [\textbf{Tierno's theorem}]\label{corotierno} A non-empty set $F$ of $\reals^{n}$  is a $d$-dimensional $C^{1}$-manifold  if and only if $F$ is locally compact  and the upper tangent and  upper paratangent cones to $F$ coincide and  are $d$-dimensional vector spaces \footnote{The $d$-dimensionality condition cannot be dropped in \eqref{tiernoC}. In fact, define $F:=\{x\in\reals^{n}:$ either $\|x\|=0$ or $\frac{1}{\|x\|}\in\naturals \}$. Then $\Tan^{+}(F,0)=\pLTan^{+}(F,0)=\reals^{n}$; moreover, for every $x\in F$ with $\|x\|\ne0$, one has $\Tan^{+}(F,x)=\pLTan^{+}(F,x)=\{v\in\reals^{n}:\langle v,x\rangle=0\}$ and $\dim(\pLTan^{+}(F,x))=n-1$. Notice that $F$ is a $C^{1}$-manifold only at every $x\ne0$.} at every point , i.e.,
\begin{equation}\label{tiernoC}
\Tan^{+}(F,x)=\pLTan^{+}(F,x)\text{ and } \dim(\LTan^{+}(F,x))=d\text{ for every }x\in F.
\end{equation}
 
\end{theorem}

This theorem provides an efficacious test for  visual geometrical reconnaissance of $C^{1}$-manifolds. In fact, it follows that $F$ is a $d$-dimensional $C^{1}$-manifold  if and only if \begin{enumerate}
\item\labelpag{tiernoTest}  at  every point of $F$, the upper tangent vectors to $F$   form a $d$-dimensional vector space which is paratangent  in traditional sense to $F$.\end{enumerate}
In symbols, by Proposition \ref{propparatangtrad-anal} this condition becomes
\begin{eqnarray}\label{tiernoTest2}
&\Tan^{+}(F,x)=\LTan^{+}(F,x),\quad \dim(\LTan^{+}(F,x))=d\quad\text{ and}\\
&\pTan^{+}(F,x)\subset \LTan^{+}(F,x)\quad\text{for every  }x\in F.\qquad\quad\nonumber
\end{eqnarray}

\begin{proof}[\textbf{\emph{Proof}}] \emph{Necessity}. By Proposition \ref{proppre}, it is obvious. \emph{Sufficiency}.  By Theorem \ref{theomain}  it is enough to show that  $\pTan^{-}(F,x)=\pTan^{+}(F,x)$ for every $x\in F$.  The first equality in $\eqref{tiernoC}$ means:
\[\Tan^{+}(F,x)=\LTan^{+}(F,x)=\pTan^{+}(F,x)=\pLTan^{+}(F,x)
\leqno(*1)\]
for every $x\in F$. Hence, by \eqref{coll2} and \eqref{coll4} we have 
\[
\pLTan^{+}(F,x)\supset\Ls_{F\ni y\to x}\pLTan^{+}(F,y)
\leqno(*2)\]
and \[
\pTan^{-}(F,x)=\Li_{F\ni y\to x}\pLTan^{+}(F,y),
\leqno(*3)\] 
for every $x\in F$.  By $\eqref{tiernoC}$ the vector spaces $\pLTan^{+}(F,x)$ and $\pLTan^{+}(F,y)$ have the same dimension; hence, from $(*2)$ and Corollary \ref{coroconv2} it follows that
\[\pLTan^{+}(F,x)=\Li_{F\ni y\to x}\pLTan^{+}(F,y).
\leqno(*4)\]
Therefore, by  $(*3)$ and $(*4)$ we obtain that $\pTan^{-}(F,x)=\pTan^{+}(F,x)$ for every $x\in F$, as required. \end{proof}

\appendix

\section{Von Neumann and  alternative definitions of lower tangent  cones}\label{app-von}

Five years before the rediscovery of the upper tangent cone by \textsc{Bouligand} and \textsc{Severi}, in \cite[(1929)]{vonN} \textsc{von Neumann} \footnote{\textsc{von Neumann}'s manuscript was received by Mathematische Zeitschrift  February 2, 1927.}
  showed that a closed matrix group $\calG$ is a Lie group by describing the associated Lie algebra (called \emph{Infinitesimalgruppe})  as  the set of all upper tangent vector at unit $E$ to the group $\calG$. The elements of $\calG$ are non-singular real matrices $n\times n$; hence, being $\calG$  a subset of Euclidean space $M_{n}(\R)$ of  all  real matrices $n\times n$, the upper tangent vectors are elements of $M_{n}(\R)$.

More explicitly and clearly,  \textsc{von Neumann} define an  \emph{Infinitesimalgruppe} $\calJ$ of $\calG$ as the set of all matrices $V\in M_{n}(\R)$ such that there exist an infinitesimal sequence $\{\varepsilon_{m}\}_{m}\subset\reals_{++}$ and a sequence $\{A_{m}\}_{m\in\naturals}\subset\calG$ such that 
\begin{equation}\lim_{m\to\infty}\frac{A_{m}-E}{\varepsilon_{m}}=V.\label{vonUpper}
\end{equation}

Moreover, to show that the Inifinitesimalgruppe $\calJ$ is a Lie algebra, \textsc{von Neumann} proved that, for every matrix $V$ belonging  to the {Infinitesimalgruppe} $\calJ$, there exists a family of matrices  $\{B_{\lambda}\}_{\lambda\in (0,1]}\subset M_{n}(\R)$ such that
\begin{equation}\lim_{\lambda\to0^{+}}\frac{B_{\lambda}-E}{\lambda}=V\label{vonLower1}.
\end{equation}
It is well known that vectors $V$ verifying \eqref{vonLower1}, constitute the lower tangent cone $\Tan^{-}(\calG,E)$.   Therefore, the definition \eqref{vonUpper} and property \eqref{vonLower1} can be resumed by
\begin{equation}\calJ:=\Tan^{+}(\calG, E)\text{ and }\Tan^{+}(\calG, E) =\Tan^{-}(\calG, E). \label{conRes}
\end{equation}

Crucial properties of general Lie groups (as ``the infinitesimal group $\calJ$ is mapped into $\calG$ by $\exp$'' or ``some neighborhood of  $E$ in $\calG$ is mapped into the infinitesimal group $\calJ$ by $\log$'') are verified by \textsc{von Neumann} by the following immediate consequence of \eqref{vonLower1}: for every $V\in \calJ$, there exists a sequence $\{A_{m}\}_{m\in\naturals}$ such that
\begin{equation}\lim_{m\to\infty}m(A_{m}-E)=V\label{vonLower2}.
\end{equation}
Tangent vectors in this sense are lower tangent, and conversely. In fact
\begin{proposition}\label{propLower}  Let $F$ be a subset of $\reals^{n}$ and let $x\in F$. Then 
\begin{equation}
\Tan^{-}(F,x)=\Li_{\naturals\ni m\to\infty}m(F-x).
\end{equation}

\end{proposition}
In terms of sequences, 
 $v\in \Tan^{-}(F,x)$ if and only if there exists a sequence $\{x_{m}\}_{m\in\naturals}\subset F$ (converging to $x$) such that 
 \begin{equation}
 \lim_{m\to\infty}m(x_{m}-x)=v.
 \end{equation}

Analogously, with respect the lower paratangent cones we have

\begin{proposition}\label{propLowerP} Let $F$ be a subset of $\reals^{n}$ and let $x\in F$. Then 
\begin{equation}
\pTan^{-}(F,x)=\Li_{\naturals\ni m\rightarrow
\infty\atop F\ni y\rightarrow x}m\left( F-y\right).
\end{equation}

\end{proposition}

In terms of sequences, 
 $v\in \pTan^{-}(F,x)$ if and only if, for every sequence $\{y_{m}\}_{m}\subset F$ converging to $x$, there exists a sequence $\{x_{m}\}_{m}\subset F$ (converging to $x$) such that \[\lim_{m\to\infty}m(y_{m}-x_{m})=v.\]

The proof of Propositions \ref{propLower} and \ref{propLowerP} 
is an immediate consequence of the following lemma.

\begin{lemma}Let $F\subset\reals^{n}$, $x\in F$ and $\Phi:F\to \calP(\reals^{n})$. Then
\begin{equation}
\Li_{\reals\ni\lambda\to+\infty\atop F\ni y\to x}\lambda\Phi(y)=\Li_{\naturals\ni m\to\infty\atop F\ni y\to x}m\Phi(y)
\end{equation}
\end{lemma}

\begin{proof}[\emph{\textbf{Proof}}] The set inclusion $\subset$ is obvious. For proving the converse set inclusion, choose an arbitrary element $v\in \Li_{\naturals\ni m\to\infty\atop F\ni y\to x}m\Phi(y)$. By the definition of lower limit we have that
\begin{equation}\label{lemmaproof1}
\lim_{\naturals\ni m\to\infty\atop F\ni y\to x}\dist(v,m\Phi(y))=0
\end{equation}
As usual, for every real number $\lambda$ let $\lfloor \lambda\rfloor$ denote the least integer number greater than or equal to $\lambda$. Observe that $\lim_{\lambda\to+\infty}\frac{\lambda}{\lfloor \lambda\rfloor}=1$, because $|\lambda-\lfloor \lambda\rfloor|\le 1$. Therefore, by \eqref{lemmaproof1} and the following triangular inequality
\begin{eqnarray}
&\dist(v,\lambda \Phi(y))=\lambda\dist(\dfrac{v}{\lambda}, \Phi(y))\le \lambda\left(\left\|\dfrac{v}{\lambda}-\dfrac{v}{\lfloor\lambda\rfloor}\right\|+\dist(\dfrac{v}{\lfloor\lambda\rfloor}, \Phi(y))\right)\\
&\le\|v\|\left(1-\dfrac{\lambda}{\lfloor \lambda\rfloor}\right)+\dfrac{\lambda}{\lfloor\lambda\rfloor}\dist(v,\lfloor \lambda\rfloor\Phi(y)),
\nonumber
\end{eqnarray}
we have 
\begin{equation}
\lim_{\reals\ni \lambda\to\infty\atop F\ni y\to x}\dist(v,\lambda\Phi(y))=0,
\end{equation}
 that is $v\in \Li_{\reals\ni\lambda\to+\infty\atop F\ni y\to x}\lambda\Phi(y)$, as required. \end{proof}

\begin{example}\label{appEx1}
\emph{Let $\{x_{m}\}_{m}\subset\reals_{++}$ be an  infinitesimal decreasing sequence.  Posit $S:=\{x_{m}:m\in\naturals\}$. Then 
\begin{equation}\label{r1}1\in\Tan^{-}(S,0)\iff \exists {\left\{m_{k}\right\} _{k}\subset \mathbb{N}%
} \text{ such that } \lim_{k\to\infty}\dfrac{x_{m_{k}}}{1/k}=1;
\end{equation}
\begin{equation}\label{r2}1\in\Tan^{-}(S,0)\iff\lim_{m\to\infty}\dfrac{x_{m+1}}{x_{m}}=1. \footnote{To prove \eqref{r2} observe that  $1\in\Tan^{-}(S,0)$ amounts to $\lim_{\lambda\to0^{+}}\frac{\dist(\lambda,S)}{\lambda}=0$. Hence, for $\lambda_{m}:=\frac{x_{m}+x_{m+1}}{2}$, $\lim_{m\to\infty}\frac{\dist(\lambda_{m},S)}{\lambda_{m}}=0$. Since $\frac{\dist(\lambda_{m},S)}{\lambda_{m}}=\frac{x_{m}-x_{m+1}}{x_{m}+x_{m+1}}$ and  $0\le \frac12(1-\frac{x_{m+1}}{x_{m}})\le \frac{x_{m}-x_{m+1}}{x_{m}+x_{m+1}}$, we have $\lim_{m\to\infty}\frac{x_{m+1}}{x_{m}}=1$, as required.  
Conversely, assume $\lim_{m\to\infty}\frac{x_{m+1}}{x_{m}}=1$. 
For every $k\in\naturals$, define $m_{k}:= \max\{m\in\naturals: \frac{1}{k}\le x_{m}\}$. Then $\lim_{k\to\infty}\dfrac{x_{m_{k}}}{1/k}=1$. Hence \eqref{r1} entails $1\in\Tan^{-}(S,0)$.
}
\end{equation}
In particular, for $x_{m}:=m!$, $\Tan^{-}(S,0)=\{0\}$ (see Example 2.1  in \textsc{Dolecki}, \textsc{Greco} \cite[(2011),\,p.\,305]{DG-Peano}).}\hfill$\Box$\end{example}

The following two examples were commented in \textsc{Dolecki}, \textsc{Greco} \cite[(2011), p. 305]{DG-Peano2}.

\begin{example}\label{appEx2}\emph{ Let $A$ be denote the set $\{(t,t\sin(1/t)): t\in\reals\setminus\{0\}\}$. Then $\Tan^-(A,(0,0))=\{(h,k)\in\R^2:|k|\le |h|\}.$
\footnote{Fix  $|\alpha|\le 1$ and let $\theta$ be $\arcsin[\alpha]$. Then
$(x_{m},y_{m}):=(\frac{1}{\theta+2m\pi},\frac{1}{\theta+2m\pi}\sin(\frac{1}{\frac{1}{\theta+2m\pi}}))=(\frac{1}{\theta+2m\pi},\frac{1}{\theta+2m\pi}\alpha)\in A$
for every natural number $m$. Since
$\lim_{m}m(x_{m},y_{m})=(\frac{1}{2\pi}, \frac{1}{2\pi}\alpha)
$,
by Proposition \ref{propLower} we have 
\[(1,\alpha)\in\Tan^{-}(A,(0,0)).
\leqno(*)\]
Analogously, 
$(x_{m},y_{m}):=(-\frac{1}{\theta+2m\pi},-\frac{1}{\theta+2n\pi}\sin(-\frac{1}{\frac{1}{\theta+2m\pi}}))=(-\frac{1}{\theta+2m\pi},\frac{1}{\theta+2m\pi}\alpha)\in A 
$
for every natural number $m$. Since
$\lim_{m}m(x_{m},y_{m})=(-\frac{1}{2\pi}, \frac{1}{2\pi}\alpha)
$
by Proposition \ref{propLower} we have  
\[(-1,\alpha)\in\Tan^{-}(A,(0,0)).
\leqno(**)\]
Now, from $(*)$ and $(**)$ if follows that
$\Tan^-(A,(0,0))\supset\{(h,k)\in\R^2:|k|\le |h|\}$;
the opposite inclusion is due to the fact that
$\Tan^-(A,(0,0))\subset\Tan^+(A,(0,0))=\{(h,k)\in\R^2:|k|\le |h|\}$.
}
}\hfill$\Box$\end{example}

\begin{example}\label{3}\emph{Let $B$ be denote the set $\{(t,-t): t\in\reals,\, t<0\}\cup\{(1/m,1/m): m\in\naturals_{1}\}$. Then $\Tan^-(B,(0,0))=\{t(1,1):t\in\R_+\}\cup\{t(-1,1):t\in\R_+\}.$
\footnote{
By applying Proposition \ref{propLower} to two sequences $\{\frac{1}{m},\frac{1}{m})\}_m$ e $\{-\frac{1}{m},\frac{1}{m})\}_m$  we have that $(\pm 1,1)\in \Tan^-(B,(0,0))$. Hence $\Tan^-(B,(0,0))=\{t(1,1):t\in\R_+\}\cup\{t(-1,1):t\in\R_+\}$, since $\Tan^-(B,(0,0))\subset \Tan^+(B,(0,0))=\{t(1,1):t\in\R_+\}\cup\{t(-1,1):t\in\R_+\}$.
}
}\hfill$\Box$\end{example}

\section{From Fréchet problem to modern characterizations of smooth manifold}\label{app-history}

In  
\cite[(1887), vol.\,III p.\,587]{jordanB1} \textsc{Jordan} defines  a curve as a continuous image of an interval.  By means of a notion of rectifiability, \textsc{Jordan} gives mathematical concreteness and coherence to the usage of the term ``length'' and, moreover, by parametrization of sets he provides fresh impetus to the study of local and global properties of sets. 

Surprisingly for \textsc{Jordan}'s epoch,  continuous curves  did not fit to common intuition on 1-dimensionality and null area of their  \emph{loci}
. In fact, \textsc{Peano} in \cite[1890]{peano_curva} constructed  a  continuous curve filling a square. Clearly, \textsc{Peano}'s curve is not simple. An example  of a simple continuous curve of non-null area was given by 
 \textsc{Lebesgue} \cite[(1903)]{lebesgue_curva} and by \textsc{Osgood} \cite[(1903)]{osgood_curva}.  \textsc{Nalli} \cite[(1911)]{nalli1,nalli2} characterized  the locus  of simple continuous plane curves by means of \emph{local connectedness} (a new notion, introduced by  \textsc{Nalli}). Three years later,  \textsc{Mazurkiewicz} \cite[(1914)]{mazurkiewicz} and \textsc{Hahn} \cite[(1914)]{hahn} proved the celebrated theorem: ``A set of Euclidean space is a continuous image of a compact interval if and only if it is a locally connected continuum''.

 In absence  of differentiable properties, the continuity alone does not capture intuitive curve aspects. Aware of this lack, 
 to recover geometric properties of the locus of a continuous curve,  
 \textsc{Fréchet} (see \cite[(1925), p.\,292-3]{frechet1925} and \cite[(1928), p.\,152-154]{frechet1928})  proposed the following \emph{problem}: \emph{Find a non-singular parametric representation\,\footnote{Here and in the sequel, ``\emph{non-singular parametric representation}'' stands for ``differentiable parametric representation with  everywhere non-null derivative''.} of the locus  of a continuous curve having tangent straight-line at  every point}. Let's quote \textsc{Fréchet} from the first reference:
\begin{quotation} On sait qu'une courbe continue sans point multiple et ayant une tangente déterminée en chaque point peut avoir une représentation paramétrique constituée de fonctions dérivables $x(t)$, $y(t)$, $z(t)$, mais dont les dérivées peuvent exceptionnellement s'annuler à la fois [\dots] 

Ce qui précède nous encourage à proposer la question suivante, dont la solution à première vue ne paraît  pas douteuse:

\emph{Si une courbe continue est douée partout \emph{(ou en un point)} d'une tangente, peut-on la représenter paramétriquement par des fonctions dérivables partout \emph{(ou au point correspondant)}?} Bien entendu, dans cet énoncé, la tangente est définie géométriquement, c'est-à-dire comme limite d'une corde.

\end{quotation}

\textsc{Fréchet}'s confidence about a solution to his problem 
was deluded in 1926 by \textsc{Valiron} \cite[(1927)]{valiron_curva}. After making precise and explicit the meaning of both \emph{tangent half-straight-line}  and \emph{tangent straight-line},  \textsc{Valiron} gives the following proposition.

\begin{theorem}[\textsc{Valiron} {\cite[(1927), p.\,47]{valiron_curva}}]\label{theo-valiron1} If  a continuous curve  admits a continuously variable oriented tangent straight-line\,\footnote{Let us express definitions given by \textsc{Valiron} by means of vectors. Let $\gamma:I\to\reals^{n}$ be a continuous curve on an open interval $I$ of real numbers. For a given parameter $t\in I$, a half straight-line issued from $\gamma(t)$ along an unit vector $v(t)$, is  said to be an \emph{oriented tangent half-straight-line}, if $v(t)=\lim_{h\to0^{+}} \frac{\gamma(t+h)-\gamma(t)}{\|\gamma(t+h)-\gamma(t)\|}$. Moreover, if  $v(t)=\lim_{h\to0^{+}} \frac{\gamma(t+h)-\gamma(t)}{\|\gamma(t+h)-\gamma(t)\|}=\lim_{h\to0^{+}} \frac{\gamma(t)-\gamma(t-h)}{\|\gamma(t)-\gamma(t-h)\|}$, the straight-line through $\gamma(t)$ along the unit vector $v(t)$ is said to be  an \emph{oriented tangent straight-line}. If $\gamma$ admits  an {oriented tangent straight-line} at every point and the map $t\mapsto v(t)$  
is continuous,  then \textsc{Valiron} says that $\gamma$ admits a \emph{continuously variable oriented tangent straight-line}. 
} at its points, then it has 
a  non-singular continuously differentiable parametric representation.

\end{theorem}

\textsc{Valiron} \cite[(1926)]{valiron_superficie} provides an analogous proposition for surfaces of ordinary $3$-dimensional space. To attain this aim, he introduces the concept of \emph{oriented tangent plane} to a surface $F$ \footnote{Using terminology of Section \ref{sec-tan-paratan-diff},  a $2$-dimensional vector space $H$ is an \emph{oriented tangent plane} to a set $F$ at a point $x$, if $H$ is equal to  the upper tangent cone to $F$ at $x$. In other words, $F$ admits an oriented tangent plane at a point $x$ if and only if the upper tangent cone at $x$ is a $2$-dimensional vector space.}, takes into account continuously turning oriented tangent plane and, in addition,  adopts the following condition at every point $x\in F$:

\begin{enumerate}
\item\labelpag{condval1}({\textsc{Valiron} \cite[(1926), p.\,190]{valiron_superficie}}) \emph{The orthogonal projection on the oriented tangent plane to $F$ at $x$ is injective on an open neighborhood of $x$ in $F$.}\end{enumerate}

\begin{theorem}[\textsc{Valiron} {\cite[(1926)]{valiron_superficie}}]\label{theo-valiron2} Let $F\subset\reals^{3}$ be  homeomorphic to a $2$-dimensional open connected set. If $F$   admits a continuously variable  oriented tangent plane  and  the  condition \eqref{condval1} holds, then $F$  locally coincides with the graph of  a continuously differentiable function.
\end{theorem}

Following \textsc{Pauc}'s counterexample \cite[(1940), p.\,96]{pauc} to Fréchet problem,  \textsc{Choquet}, in his thesis \cite[(1948), p.\,170]{choquet}, provides  necessary and sufficient conditions for the \textsc{Fréchet} supposition to hold. As \textsc{Valiron}, \textsc{Pauc} and \textsc{Choquet} make precise and explicit the meaning of tangent straight-line.  Besides, \textsc{Choquet} considers  a more general problem: \emph{If a variety admits a linear tangent variety at every point (or has  certain regularity), is there a regular parametrization (or a parametrization  having an analogous degree of regularity)?} In this spirit, \textsc{Zahorski} and \textsc{Choquet} proves the following two propositions.

\begin{proposition}[\textsc{Zahorski}, {see \textsc{Choquet} \cite[(1947), p.\,173-174]{choquet}}] If   a continuous arc  admits a tangent straight-line at  all but (possibly) countably many  points, then it has 
a   differentiable parametric representation.
\end{proposition}

\begin{proposition}[\textsc{Choquet} {\cite[(1947), p.\,174]{choquet}}] A continuous image of a compact interval is a rectifiable curve if and only if it admits a Lipschitzian   differentiable parametric representation.
\end{proposition}

Invoking seminal papers of \textsc{Fréchet} \cite[(1925)]{frechet1925},
and \textsc{Valiron} \cite[(1926, 1927)]{valiron_superficie,valiron_curva},  \textsc{Severi}  looks for non-singular continuously differentiable parametric representations
of a curve (resp.\,\,surface).  Main ingredients of the solutions of \textsc{Severi}   are \emph{strict differentiability} and \emph{paratangency}. 
 Strict differentiability ensures that  curves (resp.\,\,surfaces) have a continuously turning tangent straight line (resp.\,\,plane); it is geometrically characterized in terms  of paratangency (see Proposition \ref{propparatangtrad-diff}).  On the other hand, aware of the need of  \textsc{Valiron}'s condition \eqref{condval1}, \textsc{Severi}  assumes the following \emph{simplicity condition} and, consequently, ensures \textsc{Valiron}'s condition by  replacing  \textsc{Valiron}'s oriented tangent plane   by a paratangent plane.   
\begin{definition}[\textsc{Severi} {\cite[(1929), p.\,194]{Severi-para}, \cite[(1930), p.\,216]{Severi-curve-intuitive}, \cite[(1931), p.\,341]{Severi-Krak}, \cite[(1934), p.\,194]{Severi-diff}}]\label{condsevsim}  A $d$-dimensional topological manifold $F$ of $\reals^{n}$ satisfies the \emph{Severi simplicity 
 condition}, if the dimension of the linear hull of the  upper paratangent cone to $F$ at  every point is    
 at most $d$. 

\end{definition}

\begin{theorem}[{\textsc{Severi}  \cite[(1934), p.\,194, 196]{Severi-diff}}]\label{theo-severi}
 If $F$ is   a topological manifold of dimension one (resp. two) satisfying Severi  simplicity condition, then  the upper paratangent cone at every point is a one (resp. two) dimensional vector space 
which varies continuously.
\end{theorem}

According to \eqref{condval1} we consider the 
\begin{enumerate}
\item\labelpag{condval} {\emph{Valiron condition} for a set $F\subset\reals^{n}$}: 
For every point $x\in F$, the orthogonal projection  on the linear upper tangent space $\LTan^{+}(F,x)$  is  injective on an open neighborhood of $x$ in $F$.
\end{enumerate}

 The following theorem extends \textsc{Valiron}'s Theorem \ref{theo-valiron2}.
 
\begin{theorem}[\textbf{Valiron theorem}]\label{theomain2}
A  non-empty subset  $F$ of $\reals^{n}$ is  a $d$-dimensional $C^{1}$-manifold  if and only if $F$ is a $d$-dimensional topological manifold,   Valiron condition $\eqref{condval}$   
is satisfied and

\begin{enumerate}
\item\labelpag{main2}   $x\mapsto\Tan^{+}(F,x)$ is a continuous map from $F$ to $\GG(\reals^{n},d)$, i.e., 
\[\pTan^{-}(F,x)=\LTan^{+}(F,x)\text{ and } \dim(\LTan^{+}(F,x))=d\text{ for every }x\in F.\, \footnote{The $d$-dimensionality condition cannot be dropped in \eqref{main2}, as it is shown by the following example. Consider the set $F:=\{(x,y,z)\in\reals^{3}:\,(x^{2} + y^{2} + z^{2})^{2} = 4(x^{2} + y^{2})\}$. $F$ is a torus generated by turning the circle $S:=\{(0,y,z): (y-1)^{2}+z^{2}=1\}$   about the $z$-axis. Since the circle $S$ is tangent to $z$-axis at $(0,0,0)$, the set $F$ is a $2$-dimensional $C^{1}$-manifold  at every point different from $(0,0,0)$; while $\pTan^{-}(F,(0,0,0))=\Tan^{+}(F,(0,0,0))=\LTan^{+}(F,(0,0,0))=\reals e_{3}$ and, consequently, $\dim(\LTan^{+}(F,(0,0,0)))=1$. }\quad\Box
\]
\end{enumerate} 

\end{theorem}

The following theorem extends \textsc{Severi}'s Theorem \ref{theo-severi}, involving a simplicity condition. Severi simplicity condition $\ref{condsevsim}$ can be restated as
\begin{equation}
\label{main3}   \dim(\pLTan^{+}(F,x))\le d\text{ for every }x\in F, 
\end{equation} 
or, equivalently, 
\begin{enumerate}
\item\labelpag{Sev2} at every $x\in F$ there exists a $d$-dimensional vector space which is paratangent  in traditional sense to $F$.
\end{enumerate}

\begin{theorem}[\textbf{Severi theorem}]\label{theomain3}
A  non-empty subset  $F$ of $\reals^{n}$ is  a $d$-dimensional $C^{1}$-manifold  if and only if $F$ is a $d$-dimensional topological manifold and  Severi simplicity condition holds. 
\hfill$\Box$
\end{theorem}

 The following  corollary of Theorem \ref{theomain2} (Valiron theorem) is due to \textsc{Gluck}. 

\begin{corollary} [\textsc{Gluck} {\cite[(1966), p.\,199, 202]{gluck1} and \cite[(1968), p.\,45]{gluck2}}] A non-empty set $F\subset \reals^{n}$  is a $d$-dimensional $C^{1}$-manifold  if and only $F$ is a $d$-dimensional topological manifold and there exists a continuous map $\LTan:F\to\GG(\reals^{n},d)$ such that, for every $x\in F$,
\begin{enumerate}
\item\labelpag{condgluck2} $\LTan(x)$ is tangent in traditional sense to $F$ at $x$,
\item\labelpag{condgluck3} the orthogonal projection of $F$ on  $\LTan(x)$ is injective on an open neighborhood  of $x$ in $F$.

\end{enumerate}
\end{corollary}
\begin{proof}[\textbf{\emph{Proof}}]  \emph{Necessity}. By Proposition \ref{proppre} it is obvious. \emph{Sufficiency}. Let $\hat{x}\in F$.  By \eqref{condgluck2} and Proposition \ref{proptangtrad-anal},   
\[
\Tan^{+}(F,\hat{x})\subset \LTan(\hat{x}).
\leqno(*1)\]
On  the other hand,  being  both $F$  and $\LTan(F,\hat{x})$ $d$-dimensional topological manifolds, by Brouwer domain invariance theorem and \eqref{condgluck3}
 the orthogonal projection of $F$ into $\hat{x}+\LTan(\hat{x})$  map $\Omega$ onto an open neighborhood of $\hat{x}$ in  $\hat x+\LTan(\hat{x})$; hence,
\[
\Tan^{+}(F,\hat{x})=\LTan(\hat{x}).
\leqno(*2)\]
Hence, being $\hat{x}$ an arbitrary point of $F$, the maps $x\to \Tan^{+}(F,x)$ and $x\to \LTan(F,x)$  are equal. Therefore, applying Theorem \ref{theomain2}, we have that $F$ is a $C^{1}$-manifold, as required. \end{proof}

In \cite{gluck1} and \cite{gluck2} \textsc{Gluck} shows a very elaborated and stimulating characterization of $C^{1}$-manifold which is based on ``secant map'' and  ``shape function''. Gluck's characterization can be proved by Theorem \ref{theomain3} (Severi theorem); however,   we will prove only its unidimensional instance, since introducing the ``shape function'' is not an immediate task.

\begin{corollary} [\textsc{Gluck} {\cite[(1966), p.\,200]{gluck1} and \cite[(1968), p.\,33]{gluck2}}]\label{coroGluck}
Let $F$ be a one-dimensional topological manifold of $\reals^{n}$. Then $F$ is a $C^{1}$-manifold if and only if the function $\Sigma$ (called \emph{secant map}) from $(F\times F)\setminus \{(x,x):x\in F\}$ to $\GG(\reals^{n},1)$ which assigns to each pair  $x,y$  of distinct points of $F$ the unidimensional vector space generated by $x-y$, admits a continuous extension  over all $F\times F$. 
\end{corollary}

\begin{proof}[\textbf{\emph{Proof}}] \emph{Necessity}. Assume $F$ is a $C^{1}$-manifold of $\reals^{n}$. In order to have  the required  extension,  it is enough to  assign $\Tan^{+}(F,x)$ to every pair $(x,x)$.  \emph{Sufficiency}. Let $\pLTan(F,x)$ denote the value of the extension of  $\Sigma$ at $(x,x)$. Clearly $\pLTan(F,x)$ is a one-dimensional vector space which is paratangent in traditional sense at $x$. Therefore,  by Theorem \ref{theomain3} $F$ is a one-dimensional $C^{1}$-manifold. 
\end{proof}

Another  modern characterization of $C^{1}$-manifold is due to \textsc{Shchepin} and \textsc{Repov\v{s}} (2000); by four-cones theorem it can be rigorously proved.

\begin{corollary} [\textsc{Shchepin and Repov\v{s}} {\cite[(2000), p.\,2717]{repovs2} }]\label{cororepovs} A non-empty subset $F$ of $\reals^{n}$  is a $d$-dimensional $C^{1}$-manifold   if and only if $F$ is locally compact  and the upper-tangent and  upper-paratangent cones to $F$ coincide and  their linear hull is a $d$-dimensional vector space  at every point, i.e.,
\begin{equation}\label{cororepovs1}
\Tan^{+}(F,x)=\pTan^{+}(F,x)\text{ and } \dim(\LTan^{+}(F,x))=d\text{ for every }x\in F.\,\,\,\,\,\Box
\end{equation}  
\end{corollary}

\end{document}